\documentclass{article}

\usepackage[centertags]{amsmath}
\usepackage{hyperref}
\usepackage{amsfonts}
\usepackage{amssymb}
\usepackage{amsthm}
\usepackage{newlfont}
\usepackage{amscd}
\usepackage{amsmath,amscd}

\usepackage[curve,arrow,cmactex,matrix]{xy}
\usepackage{verbatim}

\usepackage{color}
\usepackage{eucal}

\usepackage{amssymb}

\newcommand{\C}{\mathcal{C}}
\newcommand{\D}{\mathcal{D}}
\newcommand{\F}{\mathcal{F}}
\newcommand{\G}{\mathcal{G}}

\newcommand{\BS}{\mathbf{S}}
\newcommand{\M}{\mathcal{M}}
\newcommand{\N}{\mathcal{N}}
\newcommand{\LL}{\mathcal{L}}
\newcommand{\R}{\mathcal{R}}

\newcommand{\LK}{\mathcal{LK}}
\newcommand{\RK}{\mathcal{RK}}

\newtheorem{thm}{Theorem}[subsection]
\newtheorem{cor}[thm]{Corollary}

\newtheorem{lem}[thm]{Lemma}
\newtheorem{prop}[thm]{Proposition}

\theoremstyle{definition}
\newtheorem{define}[thm]{Definition}

\theoremstyle{remark}
\newtheorem{rem}[thm]{Remark}
\newtheorem{example}[thm]{Example}

\DeclareMathOperator{\holim}{holim}

\DeclareMathOperator{\Top}{\mcal{S}}

\DeclareMathOperator{\Set}{Set}
\DeclareMathOperator{\Cat}{Cat}

\DeclareMathOperator{\Cob}{Cob}
\DeclareMathOperator{\semiCob}{semiCob}

\DeclareMathOperator{\ev}{ev}

\DeclareMathOperator{\Map}{Map}
\DeclareMathOperator{\Hom}{Hom}

\DeclareMathOperator{\qu}{qu}
\DeclareMathOperator{\nonu}{s}

\DeclareMathOperator{\cosk}{cosk}
\DeclareMathOperator{\sk}{sk}
\DeclareMathOperator{\op}{op}

\DeclareMathOperator{\CS}{CS}
\DeclareMathOperator{\CsS}{CsS}
\DeclareMathOperator{\QsS}{QsS}
\DeclareMathOperator{\GsS}{GsS}
\DeclareMathOperator{\inv}{inv}
\DeclareMathOperator{\Ho}{Ho}
\DeclareMathOperator{\aut}{aut}
\DeclareMathOperator{\df}{def}

\DeclareMathOperator{\Sp}{Sp}

\DeclareMathOperator{\Seg}{Seg}
\DeclareMathOperator{\K}{K}
\DeclareMathOperator{\Id}{Id}

\DeclareMathOperator{\ind}{ind}

\DeclareMathOperator{\Comp}{Comp}
\DeclareMathOperator{\fib}{fib}

\def\alp{{\alpha}}
\def\bet{{\beta}}
\def\gam{{\gamma}}

\def\sig{{\sigma}}
\def\vphi{{\varphi}}

\def\Om{{\Omega}}

\def\Del{{\Delta}}

\def\Lam{{\Lambda}}

\def\vphi{{\varphi}}

\def\lrar{\longrightarrow}

\def\hrar{\hookrightarrow}

\def\x{\stackrel}

\def\mcal{\mathcal}
\def\ovl{\overline}
\def\what{\widehat}
\def\wtl{\widetilde}
\def\bksl{\;\backslash\;}

\numberwithin{equation}{section}

\title{ Quasi-unital $\infty$-Categories }

\author{Yonatan Harpaz}

\begin{document}
\maketitle

\begin{abstract}
Inspired by Lurie's theory of quasi-unital algebras we prove an analogous result for $\infty$-categories. In particular, we show that the unital structure of an $\infty$-category can be uniquely recovered from the underlying non-unital structure once suitable candidates for units have been identified. The main result of this paper can be used to produce a proof for the $1$-dimensional cobordism hypothesis, as described in a forthcoming paper~\cite{har}.
\end{abstract}



\tableofcontents

\section*{ Introduction }
The notion of units in higher category theory carries considerably more structure then the corresponding discrete notion. Informally speaking, given an $\infty$-category $\C$ we are provided not only with a unit morphism $I_C: C \lrar C$ for every object $C \in \C$, but also with the precise way in which these are units, i.e. with explicit homotopies of the form $I_C \circ f \simeq f$ and $f \circ I_D \simeq f$ for every morphism $f: C \lrar D$. Furthermore, we are provided with higher homotopies exhibiting the inner coherence of the above data as well as its compatibility with composition of morphisms (along with all of its higher structure).

This bundle of information is encoded differently in different \textbf{models} for the theory of $\infty$-categories. In this paper we will be considering the model of \textbf{complete Segal space} developed by Rezk in his fundamental paper~\cite{rez}. A complete Segal space is a simplicial space $X$ satisfying a Segal condition and a completeness condition. Such a simplicial space $X$ determines an $\infty$-category $\C$ which can be described informally as follows:
\begin{enumerate}
\item
The $0$-simplices of $X$ correspond to \textbf{objects} of $\C$.
\item
The $1$-simplices of $X$ correspond to \textbf{morphisms} of $\C$ where the source and target of a given morphism are provided by the face maps $d_0,d_1: X_1 \lrar X_0$.
\item
The $2$-simplices of $X$ encode the composition in $\C$. In particular, we can think of a triangle $\sig \in X_2$ as encoding a homotopy from $d_0(\sig) \circ d_2(\sig)$ to $d_1(\sig)$.
\item
The $3$-simplices of $X$ provide us with \textbf{associativty} homotopies. Similarly, the spaces of $n$-simplices of $X$ for $n > 3$ provide us with higher coherence homotopies for the associativity strucutre.
\end{enumerate}

In this setting it natural to ask what do the structure maps $\rho^*:X_n \lrar X_m$ for the various $\rho:[n] \lrar [m]$ encode. This can be desicribed as follows. The higher face maps give information which is analogous to the source and target maps encoded by the face maps $[0] \lrar [1]$, i.e. they tells us to which objects, morphisms, etc. a specific piece of structure applies to. The degeneracy maps, on the other hand, have a different interpretation - they encode the \textbf{unital structure} of $\C$.

The $0$'th degeneracy map $s_0: X_0 \lrar X_1$ tells us for each object who is its identity morphism. Similarly, the two degeneracy maps $s_0,s_1:X_1 \lrar X_2$ provides us with homotopies of the form $I_D \circ f \simeq f$ and $g \circ I_C \simeq g$ for each morphism $f: C\lrar D$. The higher degeneracy maps can be interpreted as exhibiting the coherence of the unital structure with the composition and associativity structure.

The fact that the unital structure is encoded in the collection of degeneracies shows that it contains a somewhat intricate web of data. However, it also tells what we need to do in order to \textbf{forget it}: we should simply consider $X$ without the degeneracy maps, i.e. consider only the underlying \textbf{semi-simplicial space}.

A first motivation for forgetting this data comes from situations in which there are no natural choices for this vast unital structure. Such a case occurs, for example, when one is attempting to construct various \textbf{cobordism $\infty$-categories}. Suppose that we want to describe the $\infty$-category whose objects are closed $n$-manifolds and morphisms are cobordisms between them. Since cobordisms have their own automorphisms we can't simply take them as a set, but rather as the space classifying the corresponding topological automorphism groupoid. Gluing of cobordisms induces a weak composition operation on these classifying spaces.

As explained in~\cite{tft} \S\S $2.2$, this composition structure naturally leads to a \textbf{semi-simplicial space} $\semiCob_n$ satisfying the Segal condition. Such objects are referred to there as \textbf{semiSegal spaces}. In order to promote $\semiCob_n$ to a full simplicial space, one needs to understand the behaviour of \textbf{units} in these cobordism categories.

Now given an $n$-manifold $M$ there will certainly be an equivalence class of cobordisms $M \lrar M$ which are candidates for being the "identity" - all cobordisms which are diffeomorphic to $M \times I$. However it is a bit unnatural to choose any \textbf{specific one} of them. Note that even if we choose a specific identity cobordism $M \times I$ we will still have to arbitrarily choose diffeomorphisms of the form $\left[M \times I\right] \coprod_M W \cong W$ for each cobordism $W$ out of $M$ as well as many other coherence homotopies.

These choice problems can be overcome in various ways, some more ad-hoc than others, and in the end a unital structure can be obtained. In other words, $\semiCob_n$ can be promoted to a \textbf{Segal space} $\Cob_n$. However, there is great convenience in \textbf{not having to make these choices}. As claimed (but not proved) in~\cite{tft}, this unital structure is actually uniquely determined once we verify that suitable candidates for units exist.

Exploring this issue further, we see that an obvious necessary condition for a semiSegal space to come from a Segal space is that each object admits an endomorphism which is neutral with respect to composition (up to homotopy). Following Lurie's~\cite{higher-algebra} \S\S $6.1.3$, we will call such morphisms \textbf{quasi-units}. Informally, one is lead to consider the following questions:
\begin{enumerate}
\item
Given a non-unital $\infty$-category $\ovl{\C}$ in which every object admits a quasi-unit, can $\ovl{\C}$ be promoted to an $\infty$-category $\C$?
\item
If such a $\C$ exist, is it essentially unique?
\item
Given two $\infty$-categories $\C,\D$ with underlying non-unital $\infty$-categories $\ovl{\C},\ovl{\D}$, can the functor category $\C^\D$ be reconstructed from a suitable functor category $\ovl{\C}^{\ovl{\D}}$?
\end{enumerate}

In this paper we give a positive answer to the above questions. More precisely, we will construct a monoidal model category $\Comp_s$ which is a quasi-unital analouge of Rezk's complete Segal space model category $\Comp$. We will then show that the forgetful functor $\Comp \lrar \Comp_s$ fits into a Quillen equivalence between $\Comp$ and $\Comp_s$. Furthermore, we will show that this Quillen equivalence preserves the respective monoidal structures in an appropriate sense. This will yield an affirmative answer to all three questions above.

Before outlining our construction, let us explain our motivation for considering this question. As explained in~\cite{tft}, a result of this kind can be used to facilitate the construction of the cobordism categories. However, the relationship between such questions and the cobordism hypothesis goes beyond this mere added efficiency. In particular, one can actually use the result above in order to \textbf{prove} the $n=1$ case of the cobordism hypothesis (in the setting of $\infty$-categories). This application is described in~\cite{har}.

Let us now describe our approach for constructing $\Comp_s$. As explained above, when we encode the structure of an $\infty$-category in a simplicial space, what we need to do in order to remove the unital structure is to remove the degeneracy maps. This leads to the notion of a \textbf{semiSegal space}, which is defined formally in \S\S~\ref{ss-basic}. The data of a semiSegal space $X$ describes a non-unital $\infty$-category structure on its $0$'th space $X_0$.

The next step is to understand what it means for a morphism $f \in X_1$ such that $d_0(f) = d_1(f)$ to be a \textbf{quasi-unit}, i.e. to be neutral with respect to composition. This is defined formally in \S\S~\ref{ss-basic}. We shall say that a semiSegal space is \textbf{quasi-unital} if it admits quasi-units for every object. One can then phrase question $(1)$ above in terms of quasi-unital semiSegal spaces. However, in order to get any intelligent answer one should work not only with such semiSegal spaces themselves, but also with a correct notion of mappings between them. In particular, questions $(1)$ and $(2)$ should be considered together for an entire suitable $\infty$-category of quasi-unital semiSegal spaces.

A first discouraging observation is that maps of semiSegal spaces need not in general send quasi-units to quasi-units. This statement should be interpreted as follows: the structure of units is \textbf{not a mere condition}. Indeed, if this was the case one would expect the forget-the-units functor to be fully-faithful. Instead, we see that if $X,Y$ are two Segal spaces with underlying semiSegal spaces $\ovl{X},\ovl{Y}$, then a map $\ovl{f}: \ovl{X} \lrar \ovl{Y}$ has a chance of coming from a map $f: X \lrar Y$ only if it sends quasi-units to quasi-units. Hence we conclude that the collection of quasi-units should be \textbf{marked} as part of the data. The main result of this paper says that this is in fact all one needs to specify - all the additional unital structure is then essentially \textbf{uniquely determined}.

Our second observation is that instead of marking the quasi-units, one can instead mark the slightly larger collection of \textbf{invertible morphisms}. These are the morphisms composition with which induces weak equivalences on mapping spaces. A simple lemma (which we prove in \S\S~\ref{ss:qu-to-inv}) says that a map of semiSegal spaces sends quasi-units to quasi-units if and only if it sends invertible edges to invertible edges. Furthermore, the condition that an object admits a quasi-unit is equivalent to the condition that this object contains an invertible morphism out of it.

From this point of view we see that marking the invertible edges is essentially the same as marking the quasi-units. Furthermore, this alternative is much more convenient in practice. This is due to the fact that invertibility is a considerably more robust notion - for example, one does not need to check that a morphism has equal source and target before considering its invertibility. Furthermore, a morphism will stay invertible if we "deform it a little bit", i.e. the space of invertible $1$-simplices is a union of connected components of $X_1$.

Now in order to prove a result such as the one we are interested in here one might like a convenient model category in which one can consider semi-simplicial spaces for which certain $1$-simplices have been \textbf{marked}. This naturally leads to the category of \textbf{marked semi-simplicial spaces} in which we will work from \S\S~\ref{s:marked} onward.

In~\cite{rez} Rezk constructs two successive localizations of the Reedy model structure on the category of $\Top^{\Del^{\op}}$ of simplicial spaces. Our strategy in this paper we will be to mimic Rezk's constructions in the category $\Top^{\Del^{\op}_s}_+$ of marked semi-simplcial spaces. We will start in \S\S~\ref{s:marked} where we will establish the existence of a (monoidal) model structure on $\Top^{\Del^{\op}_s}_+$ which is analogous to the \textbf{Reedy model structure} on the category $\Top^{\Del^{\op}}$. We will refer to this structure as the \textbf{marked model structure}. We will then construct a Quillen adjunction (which is not an equivalence)
$$ \xymatrix{
\Top^{\Del^{\op}} \ar@<0.5ex>[r]^{\F^+} & \Top^{\Del^{\op}_s}_+ \ar@<0.5ex>[l]^{\RK^+} \\
}$$
between the Reedy model structure on $\Top^{\Del^{\op}}$ and the marked model structure on $\Top^{\Del^{\op}_s}_+$. This Quillen adjunction will be the basis of comparison between the model category $\Comp_s$ that will be constructed in this paper and the model category $\Comp$ of complete Segal spaces.

We will continue our strategy in \S~\ref{ss-marked-semiSegal} where we will localize the marked model structure in order to obtain the \textbf{semiSegal model category} $\Seg_s$. This model structure is analogous to the Segal model structure of~\cite{rez}. The fibrant objects of $\Seg_s$ will be called \textbf{marked semiSegal spaces}. We will then say that a marked semiSegal space is \textbf{quasi-unital} if each object admits an invertible edge out of it and if all invertible edges are marked. This will formalize the intuition described above regarding how to describe quasi-unital $\infty$-categories. We will denote by $\QsS \subseteq \Seg_s^{\fib}$ the full topological subcategory spanned by quasi-unital marked semiSegal spaces.

Following the footsteps of Rezk we observe that $\QsS$ itself is still not a model for the correct $\infty$-category of quasi-unital $\infty$-categories. As in the analogous case of Segal spaces, the problem is that equivalences in $\QsS$ are far too strict. To obtain the correct notion one needs to \textbf{localize} $\QsS$ with respect to a certain natural family of \textbf{Dwyer-Kan equivalences} which will be studies in \S\S~\ref{ss-fully-faithful}.

This route will lead us in \S~\ref{s-complete-semiSegal} to further localize the semiSegal model structure to obtain our target model category $\Comp_s$ whose fibrant objects are referred to as \textbf{complete marked semiSegal spaces}. We will denote the full topological subcategory of complete marked semiSegal spaces by $\CsS$. It is immediate to verify that complete marked semiSegal spaces are always quasi-unital, and so we have an inclusion $\CsS \subseteq \QsS$. The purpose of \S~\ref{s-complete-semiSegal} is to prove the following (see Theorems~\ref{t:localization} and~\ref{t:main-2}):

\begin{thm}\label{t:main}\
\begin{enumerate}
\item
The inclusion $\CsS \subseteq \QsS$ admits a left adjoint $\what{\bullet}: \QsS \lrar \CsS$ (in the $\infty$-categorical sense), known as the \textbf{completion functor}, which exhibits $\CsS$ as the left localization of $\QsS$ with respect to Dwyer-Kan equivalences.
\item
The adjunction $\F^+ \vdash \RK^+$ from above descends to a (suitably monoidal) \textbf{Quillen equivalence} between $\Comp$ and $\Comp_s$. In particular, the topological category $\CsS$ is equivalent to the topological category of complete Segal spaces, and this equivalence respects internal mapping objects.
\end{enumerate}
\end{thm}

Theorem~\ref{t:main} will then give us the desired positive answers to all three questions appearing above.

\subsection*{ Relation to other work }
The theory developed here is closely related and much inspired by the theory of quasi-unital algebras introduced by Lurie in~\cite{higher-algebra} \S\S $6.1.3$. There he considers non-unital algebra objects in a general monoidal $\infty$-category $\D$. Enforcing an existence condition for quasi-units and an appropriate unitality condition for morphisms one obtains the $\infty$-category of quasi-unital algebra objects in $\D$. It is then proven that

\begin{thm}[~\cite{higher-algebra}]\label{lurie-qu}
The forgetful functor from the $\infty$-category of algebra objects in $\D$ to the $\infty$-category of quasi-unital algebra objects in $\D$ is an equivalence of $\infty$-categories.
\end{thm}

Note that if $\D$ is the monoidal $\infty$-category of spaces (with the Cartesian product) then algebra objects in $\D$ can be identified with pointed $\infty$-categories with one object. Similarly, quasi-unital algebra objects in $\D$ can be considered as pointed quasi-unital $\infty$-categories with one object. Hence we see that there is a strong link between the main result of this paper and Theorem~\ref{lurie-qu}. However, even when restricting attention to quasi-unital $\infty$-categories with one object, our result is not a particular case of Theorem~\ref{lurie-qu}. This is due to the fact that the mapping space between quasi-unital $\infty$-categories with one object does not coincide, in general, with the corresponding pointed mapping space between them.

In the context of \textbf{strict n-categories} the notion of \textbf{quasi-units} has enjoyed a fair amount of interest as well. In~\cite{koc}, Kock defines the notion of a \textbf{fair $n$-category}, which in our terms can be called a \textbf{strict quasi-unital n-category}. For $n=2$ and for a variation of the $n=3$ case Kock and Joyal have shown that a (non-strict) unital structure can be uniquely recovered (see~\cite{jk1}). In~\cite{jk2} Kock and Joyal further show that every simply connected homotopy $3$-type can be modelled by a fair $3$-groupoid (see~\cite{jk2}). The main difference between their work and the present paper is that we address the (manifestly non-strict) case of quasi-unital $\infty$-categories (or $(\infty,1)$-categories, as opposed to $(n,n)$-categories). Furthermore, our results are framed in terms of a complete equivalence between the notions of unital and quasi-unital $\infty$-categories.

\section{ Preliminaries and overview }

\subsection{ Basic definitions }\label{s-preliminaries}

Let $\Del$ denote the simplicial category, i.e., the category whose objects are the finite ordered sets $[n] = \{0,...,n\}$ and whose morphisms are non-decreasing maps. Let $\Top = \Set^{\Del^{\op}}$ denote the category of \textbf{simplicial sets} endowed with the \textbf{Kan model structure}, i.e., the weak equivalences are the maps which induce isomorphisms on all homotopy groups and fibrations are Kan fibrations.

We will refer to objects $K \in \Top$ as \textbf{spaces}. We will say that two maps $f,g: K \lrar L$ in $\Top$ are \textbf{homotopic} (denoted $f \sim g$) if they induce the same map in the homotopy category associated to the Kan model structure. A \textbf{point} in a space $K$ will mean a $0$-simplex and a \textbf{path} in $K$ will mean a $1$-simplex.

Many of the categories which we will come across will be enriched in $\Top$, so that we have mapping spaces $\Map(X,Y) \in \Top$ which carry strictly associative composition rules. We will say that an $\Top$-enriched category is \textbf{topological} if all the mapping spaces are Kan.

In this paper we will be working mostly with \textbf{symmetric monoidal simplicial model categories}. Let us gloss over the relevant definitions:
\begin{define}
Let $\M$ be a category and $\otimes: \M \times \M \lrar \M$ a symmetric monoidal product. We say that $\otimes$ is \textbf{closed} if there exists an internal mapping functor $\M^{\op} \times \M \lrar \M$, typically denoted by
$ (X,Y) \mapsto Y^X $,
together with natural maps
$ \nu_{X,Y}: Y^X \otimes X \lrar Y $
which induce isomorphisms
$$ \Map\left(Z,Y^X\right) \x{\simeq}{\lrar} \Map_{\Top}(Z \otimes X,Y) $$
for every $X,Y,Z$. These isomorphisms are sometimes referred to as the \textbf{exponential law}.
\end{define}

\begin{rem}
If $\M$ is presentable and $\otimes$ is a symmetric monoidal product then $\otimes$ is closed if and only if it preserves colimits separately in each variable. This follows from the adjoint functor theorem.
\end{rem}

\begin{define}\label{d:monoidal-adj}
Let $(\C,\otimes), (\D,\times)$ be two symmetric monoidal categories and
$ \xymatrix{
\C \ar@<0.5ex>[r]^{\LL} & \D \ar@<0.5ex>[l]^{\R} \\
}$
an adjunction. Let
$$ \alp_{X,Y}: \R(X) \otimes \R(Y) \lrar \R(X \times Y) \;\;\;\;\; u: 1_{\C} \lrar \R(1_{\D}) $$
be a lax structure on $\R$ and
$$ \bet_{Z,W}: \LL(Z \otimes W) \lrar \LL(Z) \times \LL(W) \;\;\;\;\; v: \LL(1_{\C}) \lrar 1_{\D} $$
a colax structure on $\LL$. We will say that $(\alp_{X,Y},u)$ and $(\bet_{Z,W},v)$ are \textbf{compatible} if $u$ and $v$ are adjoint and the diagrams
$$
\begin{matrix}
\xymatrix{
Z \otimes W \ar[r]\ar[d] & \R(\LL(Z)) \otimes \R(\LL(W)) \ar^{\alp_{\LL(Z),\LL(W)}}[d] \\
\R(\LL(Z \otimes W)) \ar^{\R(\bet_{Z,W})}[r] & \R(\LL(Z) \times \LL(W))\\
} &
\xymatrix{
\LL(\R(X) \otimes \R(Y)) \ar^{\bet_{\R(X),\R(Y)}}[d]\ar^{\LL(\alp_{X,Y})}[r] & \LL(\R(X \times Y)) \ar[d] \\
\LL(\R(X)) \times \LL(\R(Y)) \ar[r] & X \times Y \\
}
\end{matrix}
$$
commute (where the unnamed maps are given by the unit/counit of the adjunction $\LL \dashv \R$). An adjunction with compatible lax-colax structure is called a \textbf{lax-monoidal adjunction}. We will say that a lax-monoidal adjunction is \textbf{strongly monoidal} if $\LL$ is \textbf{monoidal}, i.e. if $\bet_{Z,W}$ and $v$ are natural isomorphisms. We refer the reader to~\cite{ss} for more details.
\end{define}

\begin{rem}
If $(\alp_{X,Y},u)$ and $(\bet_{Z,W},v)$ as above are compatible then they are completely determined by each other. In fact, for any lax structure on $\R$ there is a unique compatible colax structure on $\LL$, and vice versa.
\end{rem}

\begin{define}\label{d:compatible}
Let $\M$ be a model category with a closed symmetric monoidal product $\otimes$ such that the unit of $\otimes$ is cofibrant. We say that $\M$ is \textbf{compatible} with $\otimes$ if for every pair of cofibrations $f:X' \lrar X,g:Y' \lrar Y$ the induced map
$$ h:\left[X' \otimes Y\right] \coprod_{X' \otimes Y'} \left[X \otimes Y'\right] \lrar X \otimes Y $$
is a cofibration, and is further a trivial cofibration if at least one of $f,g$ is trivial. This condition is commonly referred to as the \textbf{pushout-product axiom}. In this case we say that $\M$ is a \textbf{symmetric monoidal model category}.
\end{define}

\begin{rem}
The definition above can be extended to the case where the unit of $\otimes$ is not necessarily cofibrant (see~\cite{hov} Definition $4.2.6$). However, since in our case the units will always be confibrant it will simplify matters for us to assume this from now on.
\end{rem}

\begin{example}
The Kan model structure on $\Top$ is compatible with the Cartesian monoidal structure.
\end{example}

\begin{define}\label{d:quillen-monoidal}
Let $(\M,\otimes),(\N,\times)$ be two symmetric monoidal model categories. A Quillen adjunction $\LL \dashv \R$ between $\M$ and $\N$ will be called lax-monoidal (resp. strongly monoidal) if it is lax-monoidal (resp. strongly monoidal) as an ordinary adjunction (see Definition~\ref{d:monoidal-adj}). A lax monoidal Quillen adjunction will be called \textbf{weakly monoidal} if the structure maps of the colax structure on $\LL$ are weak equivalences.
\end{define}

\begin{define}\label{d:simplicial}
Let $\M$ be a symmetric monoidal model category. A \textbf{simplicial sturcture} on $\M$ is a strongly monoidal Quillen adjunction
$$ \xymatrix{
\Top \ar@<0.5ex>[r]^{\LL} & \M \ar@<0.5ex>[l]^{\R} \\
}.$$
In this case we say that $\M$ is a \textbf{symmetric monoidal simplicial model category}. $\M$ then acquires a natural enrichment over $\Top$ given by
$$ \Map_\M(X,Y) \x{\df}{=} \R\left(Y^X\right) $$
and one has natural isomorphisms
$$ \Map_{\Top}\left(K,\Map_\M(X,Y)) \cong \Map_\M\left(\LL(K),Y^X\right) \cong \Map_\M(\LL(K) \otimes X,Y\right) $$
for $K \in \Top^{\op},X \in \M^{\op}$ and $Y \in \M$. When there is no room for confusion we will usually abuse notation and denote $\LL(K)$ simply by $K$.
\end{define}

\begin{rem}
A simplicial structure can be defined also for $\M$ which do not posses a symmetric monoidal structure but instead carry an action of $\Top$ which satisfies analogous conditions to those of definition~\ref{d:compatible}.
\end{rem}

\subsection{ Semi-simplicial spaces and the Reedy model structure}\label{ss-semi-simplicial}

Let $\Del_s \subseteq \Del$ denote the subcategory consisting only of \textbf{injective} maps. A \textbf{semi-simplicial set} is a functor $\Del^{\op}_s \lrar \Set$. Similarly, a \textbf{semi-simplicial spaces} is a functor $\Del^{\op}_s \lrar \Top$. The category of semi-simplicial spaces will be denoted by $\Top^{\Del^{\op}_s}$.

We will denote by $\Del^n$ the standard $n$-simplex considered as a \textbf{semi-simplicial set} (it is given by the functor $\Del^{\op}_s \lrar \Set$ represented by $[n]$). If we will want to refer to the standard simplex as a \textbf{space} (i.e., an object in $\Top$) we will denote it as $|\Del^n| \in \Top$. This notation is consistent with our notation for the \textbf{geometric realization} functor which we consider as the functor
$$ |\bullet|: \Top^{\Del^{\op}_s} \lrar \Top $$
given by the coend $|X| = \int^{\Del_s} X_\bullet \times |\Del^\bullet|$.

For a subset $I \subseteq [n]$ we will denote by $\Del^I \subseteq \Del^n$ the sub semi-simplicial set corresponding to the sub-simplex spanned by $I$. We will denote by
$$ \Sp^n = \Del^{\{0,1\}} \coprod_{\Del^{\{1\}}} \Del^{\{1,2\}} \coprod_{\Del^{\{2\}}} ... \coprod_{\Del^{\{n-1\}}} \Del^{\{n-1,n\}} \subseteq \Del^n $$
the \textbf{spine} of $\Del^n$, i.e., the sub semi-simplicial set consisting of all the vertices and all the edges between consecutive vertices.

We will occasionally abuse notation and consider $\Del^n$ as a semi-simplicial space as well (which is levelwise discrete). Orthogonally, we will sometimes consider a space $K \in \Top$ as a semi-simplicial space which is concentrated in degree zero, i.e., as the semi-simplicial space given by $K_0 = K$ and $K_n = \emptyset$ for $n > 0$.

The category $\Top^{\Del^{\op}_s}$ carries the \textbf{Reedy model structure} with respect to the Kan model structure on $\Top$ and the obvious Reedy structure on $\Del_s$. Since $\Del_s$ is a Reedy category in which all non-trivial morphisms are increasing the Reedy model structure coincides with the \textbf{injective model structure}. This is a particularly nice situation because we have a concrete description for all three classes of maps. In particular, the weak equivalences and cofibrations are defined levelwise, and fibrations are defined in terms of \textbf{matching objects}. We refer the reader to~\cite{hir} \S $15$ for more details.

Now recall the standard (non-Cartesian) symmetric monoidal product $X,Y \mapsto X \otimes Y$ on $\Top^{\Del^{\op}_s}$ defined as in~\cite{rs} \S $3$ (the definition there was made originally for semi-simplicial sets but extends immediately to semi-simplicial spaces). The unit of $\otimes$ is $\Del^0$.

\begin{rem}\label{concrete}
One can obtain an explicit description of the space of $k$-simplices in $X \otimes Y$ as follows: let $P^{n,m}_k$ denote the set of injective order preserving maps
$$ \rho: [k] \lrar [n] \times [m] $$
such that $ p_{[n]} \circ \rho: [k] \lrar [n] $ and $ p_{[m]} \circ \rho: [k] \lrar [m] $ are \textbf{surjective} (such maps are sometimes called \textbf{shuffles}). Then one has
$$ (X \otimes Y)_k = \coprod_{n,m \leq k} P^{n,m}_k \times X_n \times Y_m .$$
In particular, the set of $k$-simplices of $\Del^n \otimes \Del^m$ can be identified with the set of \textbf{all} injective order preserving maps $ [k] \lrar [n] \times [m] $.
\end{rem}

\begin{rem}
The left Kan extension functor $\LK: \Top^{\Del^{\op}_s} \lrar \Top^{\Del^{\op}}$ is \textbf{monoidal} (where $\Top^{\Del^{\op}}$ is endowed with the Cartesian structure). In particular, we have natural isomorphisms
$$ \LK(X \otimes Y) \x{\simeq}{\lrar} \LK(X) \otimes \LK(Y) $$
which exhibit $(X \otimes Y)_k \subseteq \LK(X \otimes Y)_k$ as the subspace of \textbf{non-degenerate $k$-simplices} of $\left(\LK(X) \otimes \LK(Y)\right)_k$. This implies, in particular, that the geometric realization functor $|\bullet|: \Top^{\Del^{\op}_s} \lrar \Top$ is monoidal as well.
\end{rem}

The symmetric monoidal product $\otimes$ is \textbf{closed} and the corresponding internal mapping object can be described explicitly as follows: if $X,Y$ are two semi-simplicial spaces then the mapping object $Y^X$ is given by
$$ (Y^X)_n = \Map(\Del^n \otimes X,Y) .$$

The Reedy model structure on $\Top^{\Del_s^{\op}}$ is \textbf{compatible} with $\otimes$. This can be easily verified using the explicit formula in Remark~\ref{concrete}. Furthermore, $\Top^{\Del_s^{\op}}$ admits a natural \textbf{simplicial structure} (see Definition~\ref{d:simplicial}) given by the adjunction
$$ \xymatrix{
\Top \ar@<0.5ex>[r]^{\LL} & \Top^{\Del^{\op}_s} \ar@<0.5ex>[l]^{\R} \\
}$$
where $\LL(K)$ is given by $K$ concentrated in degree $0$ and $\R(X) = X_0$. In particular, the Reedy model category $\Top^{\Del_s^{\op}}$ is a symmetric monoidal simplicial model category with respect to $\otimes$.

\subsection{ SemiSegal spaces and quasi-units }\label{ss-basic}

\begin{define}\label{d:semiSegal}
Let $X$ be a semi-simplicial space. Let $[n],[m] \in \Del_s$ be two objects and consider the commutative (pushout) diagram
$$ \xymatrix{
[0] \ar^{0}[r]\ar^{n}[d]  & [m] \ar^{g_{n,m}}[d] \\
[n] \ar_{f_{n,m}}[r] & [n+m] \\
}$$
where $f_{n,m}(i) = i$ and $g_{n,m}(i) = i+n$. We will say that $X$ satisfies the \textbf{Segal condition} if for each $[n],[m]$ as above the induced commutative diagram
$$ \xymatrix{
X_{m+n} \ar^{g_{n,m}^*}[r]\ar_{f_{n,m}^*}[d] & X_m \ar^{0^*}[d] \\
X_n \ar^{n^*}[r] & X_0 \\
}$$
is \textbf{homotopy Cartesian}. We will say that $X$ is a \textbf{semiSegal space} if it is \textbf{Reedy fibrant} and satisfies the Segal condition. Note that in that case the above square will induce a homotopy equivalence
$$ X_{m+n} \simeq X_m \times_{X_0} X_n .$$
\end{define}

\begin{example}\label{e-topological}
Let $\D$ be a small non-unital topological category. We can associate with $\D$ a semiSegal space via a non-unital analogue of the nerve construction as follows. For each $n$, let $C^{\nonu}([n])$ denote the non-unital category whose objects are the numbers $0,...,n$ and whose mapping spaces are
$$ \Map_{\C^{\nonu}([n])}(i,j) = \left\{
\begin{matrix}
\emptyset & i \geq j \\
* & i < j \\ \end{matrix}\right. $$
As $\C^{\nonu}([n])$ depends functorially on $[n] \in \Del_s$ we can get a semi-simplicial space $N(\D)$ by setting
$$ N(\D)_n = \Map_{\Cat_{\Top}}(\C^{\nonu}([n]),\D) .$$
Note that $N(\D)$ will generally not be Reedy fibrant, but after applying the Reedy fibrant replacement functor (which is a levelwise equivalence) one indeed obtains a \textbf{semiSegal space}.
\end{example}

We think of a general semiSegal space as encoding a relaxed version of Example~\ref{e-topological}, i.e. a non-unital $\infty$-category. This can be described as follows: the objects of this non-unital $\infty$-category are the points of $X_0$. Given two points $x,y \in X_0$ we define the \textbf{mapping space} between them by
$$ \Map_{X}(x,y) = \{x\} \times_{X_0} X_1 \times_{X_0} \{y\}, $$
i.e., as the fiber of the (Kan) fibration
$$ X_1 \x{(d_0,d_1)}{\lrar} X_0 \times X_0 $$
over the point $(x,y)$. The space $X_2$ of triangles then induces a weak composition operation on these mapping spaces which is homotopy associative in a coherent way. For a more detailed description in the unital case we refer the reader to~\cite{rez}.

\begin{rem}\label{r:uncomplete}
As in the unital case, a semiSegal space carries more information then just a non-unital $\infty$-category structure on $X_0$. One aspect of this is that $X_0$ itself is not a set, but a space, and the homotopy type of this space is not determined by the non-unital $\infty$-category structure. In the unital case  (as well as the quasi-unital case, as we will see in \S\S~\ref{s-complete-semiSegal}) this issue can be resolved via the notion of \textbf{completeness}.
\end{rem}

\begin{example}\label{e:cosk_0}
Let $Z$ be a Kan simplicial set. Applying the $0$'th coskeleton functor one obtains a semi-simplicial space
$$ X = \cosk_0(Z) $$
which is given by $X_n = \Map\left(\sk_0(\Del^n),Z\right) = Z^{n+1}$. It is then easy to verify that $X$ is a \textbf{semiSegal space}. This semiSegal spaces encodes a very "boring" non-unital structure in which all the mapping spaces are contractible. However, it can admit arbitrary homotopy types for the space of objects $X_0$ (see Remark~\ref{r:uncomplete}).
\end{example}

\begin{define}
Let $X$ be a semiSegal space. We define its \textbf{non-unital homotopy category} $\Ho(X)$ to be the non-unital category whose objects are the points of $X_0$ and whose morphism sets are given by
$$ \Hom_{\Ho(X)}(x,y) \x{\df}{=} \pi_0(\Map_X(x,y)) .$$
The composition is determined by $X_2$ in the following way: for each triangle $\sig \in X_2$ of the form
$$ \xymatrix{
& y \ar^{g}[dr] & \\
x \ar^{f}[ur] \ar^{h}[rr] && z \\
}$$
the component $[h] \in \pi_0(\Map_X(x,z))$ is the composition of the components $[f] \in \pi_0(\Map_X(x,y))$ and $[g] \in \pi_0(\Map_X(y,z))$. The Segal condition in dimensions $2$ and $3$  ensures that this composition is well defined and associative.
\end{define}

Our first goal when dealing with semiSegal spaces is to understand when a morphism $f:x \lrar y$ is \textbf{neutral} with respect to composition. For this we need to extract in some way the action of $f$ on mapping spaces. This can be done as follows. Let $x,y,z \in X_0$ be points and $f: x \lrar y$ a morphism in $X$ (i.e. a point $f \in \Map_X(x,y)$). Consider the space
$$ C^R_{f,z} = \left\{\sig \in X_2 \left|\; \sig|_{\Del^{\{1,2\}}} = f, \sig|_{\Del^{\{0\}}} = z \right\}\right. $$
together with the two restriction maps
$$ \xymatrix{
& C^R_{f,z} \ar^{\psi}[dr] \ar_{\vphi}[dl] & \\
\Map_{X}(z,x) & & \Map_{X}(z,y) \\
}\\
 $$
By the Segal condition we see that $\vphi$ is a weak equivalence. We then define a homotopy class
$$ [f]_* \in \Hom_{\Ho(\Top)}\left(\Map_X(z,x),\Map_X(z,y)\right) $$
by setting $[f]_* \x{\df}{=} [\psi] \circ [\vphi]^{-1}$. This can be considered as the homotopy class of the almost-defined map $f_*$ obtained by composition with $f$. Similarly, one can define a homotopy class
$$ [f]^*: \Hom_{\Ho(\Top)}\left(\Map_X(y,z),\Map_X(x,z)\right) $$
describing the homotopy class of pre-composition with $f$.

\begin{rem}
Given a point $z \in X_0$ the definition above yields a \textbf{functor} of non-unital categories
$ \Ho(X) \lrar \Ho(\Top) $
given by
$$ x \mapsto \Map_X(z,x), \;\;\; [f] \mapsto [f]_* .$$
Similarly, we can construct a functor
$ \Ho^{\op}(X) \lrar \Ho(\Top) $
by setting
$$ x \mapsto \Map_X(x,z), \;\;\; [f] \mapsto [f]^* .$$
These functors correspond to the representable and corepresentable functors of $X$ after descending to the (non-unital) homotopy category (although they are not the representables and corepresentables of $\Ho(X)$ itself as they take values in $\Ho(\Top)$ and not in $\Set$).
\end{rem}

The above construction can be used to determine when a morphism is neutral with respect to composition:
\begin{define}
Let $x \in X_0$ be an object and $f: x \lrar x$ a morphism in $X$. We will say that $f$ is a \textbf{quasi-unit} if for each $z \in X_0$ the homotopy classes
$$ [f]_* \in \Hom_{\Ho(\Top)}\left(\Map_{X}(z,x), \Map_{X}(z,x)\right) $$
and
$$ [f]^* \in \Hom_{\Ho(\Top)}\left(\Map_{X}(x,z), \Map_{X}(x,z)\right) $$
are both the \textbf{identity} in $\Ho(\Top)$.
\end{define}

\begin{define}
Let $X$ be a semiSegal space. We will say that $X$ is \textbf{quasi-unital} if every object $x \in X_0$ admits a quasi-unit $q: x \lrar x$. We will informally say that $X$ models a \textbf{quasi-unital $\infty$-category}.
\end{define}

\begin{example}
The semiSegal spaces $\semiCob_n$ constructed in~\cite{tft} \S\S $2.2$ (which model the underlying non-unital $\infty$-category of the $n$'th cobordism category) are easily seen to be quasi-unital. Informally speaking, any trivial cobordism from an $n$-manifold $M$ to itself corresponds to a quasi-unit in $\semiCob_n$.
\end{example}

\begin{rem}
If $X$ is a quasi-unital semiSegal space then $\Ho(X)$ acquires a natural structure of a \textbf{unital category}.
\end{rem}

For each $x \in X_0$, we will denote by $X^{\qu}_x \subseteq \Map_X(x,x)$ the maximal subspace spanned by the quasi-units $f \in \left(\Map_X(x,x)\right)_0$. Clearly $X^{\qu}_x$ is a union of connected components of $\Map_X(x,x)$.

\begin{lem}\label{qu-connected}
Let $X$ be a semiSegal space and $x \in X_0$ a point. If $X^{\qu}_x$ is not empty then it is connected.
\end{lem}
\begin{proof}

Let $q_1,q_2: x \lrar x$ be two quasi-units. We need to show that $q_1,q_2$ are in the same connected component of $X^{\qu}_x$. Since $X^{\qu}_x$ is a union of components of $\Map_X(x,x)$ it is enough to show that $q_1,q_2$ are in the same connected component of $\Map_X(x,x)$.
Since $q_1$ is a quasi-unit there exists a triangle of the form
$$ \xymatrix{
& x \ar^{q_3}[dr] & \\
x \ar^{q_1}[ur] \ar^{q_2}[rr] & & x \\
}$$
for some $q_3: x \lrar x$. Then $q_3$ is necessarily a quasi-unit and so $q_1$ and $q_2$ are in the same connected component of $\Map_X(x,x)$.
\end{proof}

\subsection{ From quasi-units to invertible edges }\label{ss:qu-to-inv}

Our interest in this paper is to study quasi-unital $\infty$-categories and functors between them which respect quasi-units. As explained in the introduction, it will be useful to weaken the definition of quasi-units and to consider the more robust notion of \textbf{invertible edges}:

\begin{define}\label{inv-unital}
Let $x,y \in X_0$ be two objects and $f: x \lrar y$ a morphism in $X$. We will say that $f$ is \textbf{invertible} if for every $z \in X_0$ the homotopy classes
$$ [f]_* \in \Hom_{\Ho(\Top)}\left(\Map_{X}(z,x), \Map_{X}(z,y)\right) $$
and
$$ [f]^* \in \Hom_{\Ho(\Top)}\left(\Map_{X}(y,z), \Map_{X}(x,z)\right) $$
are \textbf{isomorphisms} in $\Ho(\Top)$.
\end{define}

\begin{rem}\label{inv-concrete}
It is immediate from the definition that a morphism $f:x \lrar y$ in $X$ is invertible if and only if
\begin{enumerate}
\item
Each map of the form $ \sig:\Lam^2_2 \lrar X $ such that $\sig\left(\Del^{\{1,2\}}\right) = f$ has a contractible space of extensions $ \ovl{\sig}:\Del^2 \lrar X $.
\item
Each map of the form $ \sig:\Lam^2_0 \lrar X $ such that $\sig\left(\Del^{\{0,1\}}\right) = f$ has a contractible space of extensions $ \ovl{\sig}:\Del^2 \lrar X $.
\end{enumerate}
\end{rem}

Invertible morphisms can be described informally as morphisms composition with which induces a weak equivalence on mapping spaces. Note that the notion of invertibility does not presuppose the existence of identity morphisms, i.e., it makes sense in the non-unital setting as well.

We will denote by
$ X^{\inv}_1 \subseteq X_1 $
the maximal subspace spanned by the invertible vertices $f \in (X_1)_0$. Using Reedy fibrancy it is not hard to show that $X^{\inv}_1$ is just the \textbf{union of connected components of $X_1$} which meet invertible edges.

\begin{define}\label{d:semiKan}
Let $X$ be a semiSegal space. We will say that $X$ is a \textbf{semiKan space} if every edge in $X_1$ is invertible.
\end{define}

The following lemma follows immediately from the definition:
\begin{lem}\label{inv-closed-composition}
Let $X$ be a semiSegal space. Then the space of invertible edges satisfies the following closure properties:
\begin{enumerate}
\item (2-out-of-3)
Let $\sig: \Del^2 \lrar X$ be a $2$-simplex with two of the edges being invertible. Then the third edge is invertible as well.
\item (2-out-of-6)
Let $\sig: \Del^3 \lrar X$ be a $3$-simplex such that $\sig\left(\Del^{\{0,2\}}\right)$ and $\sig\left(\Del^{\{1,3\}}\right)$ are invertible. Then all the edges of $\sig$ are invertible.
\end{enumerate}
\end{lem}

Our next goal is to verify that for the purpose of studying quasi-unital $\infty$-categories one can replace the notion of quasi-units with that of invertible edges. We begin with the following observation:

\begin{lem}\label{l:qu-inv}
Let $X$ be a semiSegal space and $x \in X_0$ a point. Then $x$ admits a quasi-unit if and only if there exists an invertible edge with source $x$.
\end{lem}
\begin{proof}
If $x$ has a quasi-unit then this quasi-unit is in particular an invertible edge with source $x$. On the other hand, if $f: x \lrar y$ is an invertible edge then according to Remark~\ref{inv-concrete} there exist a triangle $\sig: \Del^2 \lrar X$ of the form
$$ \xymatrix{
& y & \\
x \ar^{f}[ur] \ar^{q}[rr] & & x \ar_{f}[ul] \\
}$$
Then $q$ is necessarily a quasi-unit and we are done.
\end{proof}

This means that the existence condition for quasi-units can be phrased equivalently in terms of invertible edges. Our next proposition verifies that the associated restrictions on functors are equivalent as well:

\begin{prop}\label{unital-if-marked}
Let $\vphi:X \lrar Y$ be a map between quasi-unital semiSegal spaces. The following are equivalent:
\begin{enumerate}
\item
$\vphi$ sends quasi-units to quasi-units.
\item
$\vphi$ sends invertible edges to invertible edges.
\end{enumerate}
\end{prop}
\begin{proof}
First assume that $\vphi$ sends invertible edges to invertible edges and let $x \in X_0$ a point. Since $X$ is quasi-unital there exists a quasi-unit $q: x \lrar x$. Then there must exist a triangle of the form
$$ \xymatrix{
& x & \\
x \ar^{q}[ur] \ar^{q'}[rr] & & x \ar_{q}[ul] \\
}$$
in which $q'$ is necessarily a quasi-unit as well. The map $\vphi$ then sends this triangle to a triangle of the form
$$ \xymatrix{
& \vphi(x) & \\
\vphi(x) \ar^{\vphi(q)}[ur] \ar^{\vphi(q')}[rr] & & \vphi(x) \ar_{\vphi(q)}[ul] \\
}$$
where $\vphi(q)$ is invertible, and hence $\vphi(q')$ is a quasi-unit. From Lemma~\ref{qu-connected} we get that $\vphi$ maps all quasi-units of $x$ to quasi-units of $\vphi(x)$.

Now assume that $\vphi$ sends quasi-units to quasi-units and let $f:x \lrar y$ be an invertible edge. Since $X$ is quasi-unital there exist quasi-units $q: x \lrar x$ and $r:y \lrar y$. Since $f$ is invertible there exist triangles of the form

\begin{align}
\xymatrix{
& y  \ar^{g}[dr] &\\
x \ar^{f}[ur] \ar^{q}[rr] & & x \\
} &&
\xymatrix{
& x \ar^{f}[dr] & \\
y \ar^{h}[ur] \ar^{r}[rr] & & y \\
}
\end{align}

Applying $\vphi$ to these triangles and using the fact that $\vphi(q),\vphi(r)$ are quasi-units we get that $\vphi(f)$ is invertible. This finishes the proof of Proposition~\ref{unital-if-marked}.

\end{proof}

Lemma~\ref{l:qu-inv} and Proposition~\ref{unital-if-marked} suggest that the notion of a quasi-unital $\infty$-category can be encoded as semiSegal spaces in which every object admits an invertible edge out of it. Furthermore, in order to consider only functors which respect quasi-units one can instead study maps of semiSegal space which preserve invertible edges.

At this point it is worth while to consider the particular case of \textbf{quasi-unital semiKan spaces} (see Definition~\ref{d:semiKan}). In this case every map automatically respects invertible edges (and hence quasi-units) and so we can study it without any additional technicality. It turns out that the desired result in this particular case is straightforward generalization of the well-known theorem of Graeme Segal. These ideas will be explained in the next subsection.

\subsection{ Quasi-unital semiKan spaces }\label{ss-qu-groupoids}
Let $X$ be a quasi-unital semiKan space. We will say that $X$ is \textbf{connected} if for each $x,y \in X_0$ one has $\Map_{X}(x,y) \neq \emptyset$. Every quasi-unital semiKan space is a disjoint union of connected quasi-unital semiKan spaces in an essentially unique way. We will refer to these as the \textbf{connected components} of $X$. We will say that a map $f: X \lrar Y$ of quasi-unital semiKan spaces is a \textbf{DK-equivalence} if it induces an isomorphism on the set of connected components and induces a weak equivalence on mapping spaces.

Let $ \GsS \subseteq \Top^{\Del_s}$ be the full subcategory spanned by quasi-unital semiKan spaces. The localization of $\GsS$ by DK-equivalences is a natural choice for a model for the homotopy theory of \textbf{quasi-unital $\infty$-groupoids}.

In this subsection we will study this localized category via the geometric realization functor
$ |\bullet|: \Top^{\Del^{\op}_s} \lrar \Top $.
This functor has a right adjoint
$ \Pi: \Top \lrar \Top^{\Del^{\op}_s}_+ $
given by
$$ \Pi(Z)_n = \Map_S(|\Del^n|,Z) $$
for $Z \in \Top$. When $Z$ is a Kan complex, the semi-simplicial space $\Pi(Z)$ is fibrant, and one can easily verify that it is a quasi-unital semiKan space.

Let $\K \subseteq \Top$ be the full subcategory spanned by Kan complexes. For a simplicial set $Z \in \Top$ let us denote by $\what{Z} \in \K$ the (functorial) Kan replacement of $Z$. The functor $\what{\bullet}$ is then homotopy-left adjoint to the full inclusion $\K \subseteq \Top$. Furthermore, it exhibits $\K$ as the left localization of $\Top$ with respect to weak equivalences.

From the above considerations we see that the adjoint pair
$$ |\bullet|: \Top^{\Del^{\op}_s} \rightleftarrows \Top: \Pi $$
induces a \textbf{homotopy-adjoint} pair
$$ \what{|\bullet|}: \GsS \rightleftarrows \K: \Pi|_{\K} .$$
Since the category $\Del_s$ is weakly contractible we get that the functor $\Pi_{|\K}$ is actually \textbf{fully-faithful} (i.e., induces an equivalence between $\K$ and the full subcategory of $\GsS$ spanned by its image) and the counit map is a natural equivalence. Hence we can consider $\what{|\bullet|}$ as a \textbf{left localization} functor. The class of morphisms by which it localizes are the morphisms which it sends to equivalences.

In this section we will prove that this class of equivalences localized by $|\bullet|$ are exactly the DK-equivalences (Corollary~\ref{equiv-preserves-realization} below). This means that $\what{|\bullet|}$ serves as a \textbf{left localization} functor with respect to DK-equivalences between quasi-unital semiKan spaces. We can frame this theorem as follows:
\begin{thm}\label{main-groupoid}
The $\infty$-category of quasi-unital $\infty$-groupoids (i.e., the localization of $\GsS$ by DK-equivalences) is equivalent to $\K$. The equivalence is given by sending a semiKan space $X$ to the Kan replacement of its realization $\what{|X|}$.
\end{thm}

Note that this is exactly what happens in the case of \textbf{unital} $\infty$-groupoids (where the realization functor is sometimes referred to as \textbf{classifying space}). In particular, quasi-unital and unital $\infty$-groupoids have the same homotopy theory, i.e., the homotopy theory of Kan complexes. Hence we get the main conclusion of this subsection:
\begin{thm}\label{t:main-groupoid-2}
The forgetful functor induces an equivalence between the $\infty$-category of $\infty$-groupoids and the $\infty$-category of quasi-unital $\infty$-groupoids.
\end{thm}

The core argument for proving Theorem~\ref{main-groupoid} is Theorem~\ref{realization-of-groupoid} below. The particular case of Theorem~\ref{realization-of-groupoid} where $X_0$ is a point (and simplicial sets are replaced with actual topological spaces) is a famous theorem of Segal whose proof was completed by Puppe (see ~\cite{seg} and ~\cite{pup}). The proof we offer is a straightforward adaptation of the Segal-Puppe proof to the general case.

Given a space $Z$ and points $x,y \in Z$ we will denote $\Om(Z,x,y)$ the space of paths in $Z$ from $x$ to $y$ and by $P(Z,x)$ the space of paths in $z$ which start at $x$.

\begin{thm}\label{realization-of-groupoid}
Let $X$ be a semiKan space and $x,y \in X_0$ two points. We will consider $x,y$ as points in $\what{|X|}$ via the natural inclusion $X_0 \hrar \what{|X|}$. Then the natural map
$$ \Map_X(x,y) \lrar \Om\left(\what{|X|},x,y\right) $$
is a weak equivalence.
\end{thm}
\begin{proof}

We will rely on the main result of~\cite{pup} which can be stated as follows:
\begin{thm}\label{pupe}
Let $X,Y$ be two semi-simplicial spaces and let $\vphi: X \lrar Y$ be a map such that for each $f: [k] \lrar [n]$ in $\Del_s$ the square
$$ \xymatrix{
X_n \ar^{\vphi_n}[d]\ar^{f^*}[r] & X_k \ar^{\vphi_k}[d] \\
Y_n \ar^{f^*}[r] & Y_k \\
}$$
is homotopy Cartesian. Then the square
$$ \xymatrix{
X_0 \ar[r]\ar[d] & |X| \ar[d] \\
Y_0 \ar[r] & |Y| \\
}$$
is homotopy Cartesian as well.
\end{thm}

Let us now prove Theorem~\ref{realization-of-groupoid}. Define the path semiKan space $P(X,x)$ as follows:
$$ P(X,x)_n = \{\sig \in X_{n+1} | d_0(\sig) = x\} \subseteq X_{n+1} .$$
It is not hard to see that $P(X,x)$ is also a semiKan space. We have a natural map
$ p: P(X,x) \lrar X $
which maps $\sig \in P(X,x)_n$ to $\sig|_{\Del^{\{1,...,n+1\}}} \in X_n$. Note that for each $n$ the map
$ p_n: P(X,x)_n \lrar X_n $
is a fibration whose fiber over $\sig \in X_n$ is homotopy equivalent to the mapping space
$ \Map_X\left(x,\sig|_{\Del^{\{0\}}}\right)$. Furthermore, it is not hard to see that for each $f: [k] \lrar [n]$ in $\Del_s$ the square
$$ \xymatrix{
P(X,x)_n \ar^{f^*}[r]\ar^{p_n}[d] & P(X,x)_k \ar^{p_k}[d] \\
X_n\ar^{f^*}[r] & X_k \\
}$$
is homotopy Cartesian. Hence by Puppe's theorem the left square in the diagram
$$ \xymatrix{
P(X,x)_0 \ar[r]\ar^{p_0}[d] & |P(X,x)| \ar[d]\ar[r] & P\left(\what{|X|},x\right) \ar^{\ev_1}[d] \\
X_0 \ar[r] & |X| \ar@{=}[r] & |X| \\
}$$
is a homotopy Cartesian (where $\ev_1$ is the function which associates to a path $\gam$ its value $\gam(1)$). Now note that $p_0$ and $\ev_1$ are both fibrations. Identifying the fibers of these fibrations we see that the desired result is equivalent to the exterior rectangle being homotopy Cartesian, or equivalently, that the right square is homotopy Cartesian. Since $P\left(\what{|X|},x\right)$ is contractible it will suffice to show that:

\begin{lem}\label{P-is-contractible}
The space $|P(X,x)|$ is contractible.
\end{lem}
\begin{proof}
Consider the natural map
$$ P(X,x) \lrar \cosk_0(P(X,x)_0) .$$
Unwinding the definition of $P(X,x)$ and using the fact that $X$ is a semiKan space we see that this map is actually a levelwise weak equivalence. Note that the realization of a semi-simplicial space coincides with its homotopy colimit and so is preserved by levelwise equivalences. Hence it is enough to show that
$ |\cosk_0(P(X,x)_0)| $
is contractible. This in turn is due to the fact that any semi-simplicial space of the form $\cosk_0(Z)$ for $Z \neq \emptyset$ admits a canonical semi-simplicial null-homotopy $\Del^1 \otimes \cosk_0(Z) \lrar \cosk_0(Z)$.
\end{proof}
This finishes the proof of Theorem~\ref{realization-of-groupoid}.
\end{proof}

\begin{cor}\label{c:kan-counit}
Let $X$ be a semiKan space. Then the counit map
$ X \lrar \Pi\left(\what{|X|}\right) $
is a DK-equivalence.
\end{cor}
\begin{proof}
By Theorem~\ref{realization-of-groupoid} the counit map is fully-faithful. Since the map $X_0 \lrar \what{|X|}$ is surjective on connected components we see that the map $f$ is in fact a DK-equivalence.
\end{proof}

\begin{cor}\label{equiv-preserves-realization}
Let $f: X \lrar Y$ be a map between quasi-unital semiKan spaces. Then $f$ is a DK-equivalence if and only if the induced map
$ f_*:\what{|X|} \lrar \what{|Y|} $
is a weak equivalence.
\end{cor}
\begin{proof}
First note that the connected components of $X$ as a semiKan space are in bijection with the connected components of $\what{|X|}$ as a space. Hence Theorem~\ref{realization-of-groupoid} tells us that $f: X \lrar Y$ is a DK-equivalence if and only if it induces a bijection
$$ \pi_0\left(\what{|X|}\right) \lrar \pi_0\left(\what{|Y|}\right) $$
and for each $x,y \in \what{|X|}$ the induced map
$$ \Om\left(\what{|X|},x,y\right) \lrar \Om\left(\what{|Y|},f(x),f(y)\right) $$
is a weak equivalence. But this is equivalent to $f_*:\what{|X|} \lrar \what{|Y|} $ being a weak equivalence and we are done.
\end{proof}

We finish this subsection with an application which we record for future use. Recall that in general geometric realization does not commute with \textbf{Cartesian products} of semi-simplicial spaces (i.e., levelwise products). The following corollary shows that in the specific case of semiKan spaces, geometric realization \textbf{does} commute with Cartesian products:
\begin{cor}\label{products}
Let $X,Y$ be two quasi-unital semiKan spaces. Then the natural map
$ |X \times Y| \lrar |X| \times |Y| $
is a weak equivalence.
\end{cor}
\begin{proof}
First note that if $X,Y$ are semiKan spaces then $X \times Y$ is a semiKan space as well. Furthermore, it is clear that the natural map
$$ \pi_0(|X \times Y|) \lrar \pi_0(|X|) \times \pi_0(|Y|) $$
is an isomorphism (as both sides can be identified with the set of connected components of the semiKan space $X \times Y$).

Now let $x_1,x_2 \in X, y_1,y_2 \in Y$ be points and consider the natural map
$$ \Om\left(|X \times Y|,(x_1,y_1),(x_2,y_2)\right) \lrar \Om\left(|X| \times |Y|,(x_1,y_1),(x_2,y_2)\right) .$$
Theorem~\ref{realization-of-groupoid} show that this map is weakly equivalent to the isomorphism
$$ \Map_{X \times Y}((x,y),(x,y)) \x{\simeq}{\lrar} \Map_X(x,x) \times \Map_Y(y,y) $$
and so is itself a weak equivalence. The desired result now follows.
\end{proof}

The purpose of this paper is to obtain a generalization of Theorem~\ref{t:main-groupoid-2} from $\infty$-groupoids to $\infty$-categories. This will require some modifications in order to guarantee that we consider only functors which send invertible edges to invertible edges.

In order to keep track of invertible edges it will be useful to work in a variant of the category of semiSegal spaces where the invertible edges can be somehow \textbf{marked}. For this one needs to replace the notion of a semi-simplicial space with that of a \textbf{marked semi-simplicial space}. The formal basis for such a framework will be laid out in the next subsection.

\subsection{ Marked semi-simplicial spaces }\label{s:marked}

Let us open with the basic definition:
\begin{define}
A \textbf{marked semi-simplicial space} is a pair $(X,A)$ where $X$ is a semi-simplicial space and $A \subseteq X_1$ is a subspace. In order to keep the notation clean we will often denote a marked semi-simplicial space $(X,A)$ simply by $X$. Given two marked semi-simplicial spaces $(X,A), (Y,B)$ we denote by
$$ \Map^{+}(X,Y) \subseteq \Map(X,Y) $$
the subspace of maps which send $A$ to $B$. We will refer to this kind of maps as \textbf{marked} maps.
We denote by
$ \Top^{\Del_s^{\op}}_+ $
the $\Top$-enriched category of marked semi-simplicial spaces and marked maps between them.
\end{define}

\begin{rem}
The analogous notion of marked simplicial sets plays an essential role in the theory of $\infty$-categories as developed in~\cite{higher-topos}. Our definition above as well as many of the associated notations are following their analogues in~\cite{higher-topos}.
\end{rem}

\begin{define}
Given a semi-simplicial space $X$ we will denote by $X^{\sharp}$ the marked semi-simplicial $(X,X_1)$ in which \textbf{all} edges are marked. The association $X \mapsto X^{\sharp}$ is right adjoint to the forgetful functor $(X,A) \mapsto X$.
\end{define}

\begin{define}
Given a semi-simplicial space $X$ we will denote by $X^{\flat}$ the marked semi-simplicial space $(X,\emptyset)$ in which \textbf{no} edges are marked. The association $X \mapsto X^{\flat}$ is left adjoint to the forgetful functor $(X,A) \mapsto X$.
\end{define}

\begin{define}
Let $(X,A)$ be a marked semi-simplicial space. We will denote by $\ovl{A} \subseteq \pi_0(X_1)$ the \textbf{image} of the map
$$ \pi_0(A) \lrar \pi_0(X), $$
i.e., the set of connected components of $X_1$ which meet $A$. We refer to $\ovl{A}$ as the set of \textbf{marked connected components} of $X_1$.
\end{define}

\begin{define}
We will say that a map $f: (X,A) \lrar (Y,B)$ of marked semi-simplicial spaces is a \textbf{marked equivalence} if
\begin{enumerate}
\item
The underling map $f: X \lrar Y$ is a levelwise equivalence.
\item
The induced map
$ f_*: \ovl{A} \lrar \ovl{B} $
is an isomorphism of sets.
\end{enumerate}
\end{define}

\begin{thm}\label{marked-model}
There exists a left proper combinatorial model category structure on $\Top^{\Del^{\op}_s}_+$ such that
\begin{enumerate}
\item
The weak equivalences are the marked equivalences.
\item
The cofibrations are the maps $f: (X,A) \lrar (Y,B)$ for which the underlying map $X \lrar Y$ is a cofibration (i.e., levelwise injective).
\item
A map is a fibration if and only if it satisfies the right lifting property with respect to all morphisms which are both cofibrations and weak equivalences.
\end{enumerate}
\end{thm}
\begin{proof}
We will use a general existence theorem which is a slightly weaker version of Proposition $A.2.6.13$ of~\cite{higher-topos} (which in turn is based on work of Smith). In the following the term \textbf{presentable} is used as in~\cite{higher-topos} (which in classical terminology is often called \textbf{locally presentable}).
\begin{thm}[Lurie, Smith]\label{existence}
Let $\M$ be a presentable category. Let $C,W$ be two classes of morphisms in $\M$ such that
\begin{enumerate}
\item
$C$ is weakly saturated and is generated (as a weakly saturated class of morphisms) by a \textbf{set} of morphisms $C_0$.
\item
$W$ is \textbf{perfect} (see Definition $A.2.6.10$ of ~\cite{higher-topos}).
\item
$W$ is stable under pushouts along $C$, i.e., if
$$ \xymatrix{
X \ar^{f}[r]\ar^{g}[d] & Y \ar^{g'}[d] \\
Z \ar[r] & W \\
}$$
is a pushout square such that $f \in C$ and $g \in W$ then $g' \in W$ as well.
\item
If a morphism $f$ in $\M$ has the right lifting property with respect to every morphism in $C$ (or equivalently in $C_0$) then $f \in W$.
\end{enumerate}
Then there exists a left proper combinatorial model structure on $\M$ such that the weak equivalences are $W$ and the cofibrations are $C$.
\end{thm}

First it is clear that $\Top^{\Del^{\op}_s}_+$ is presentable. Let $W$ be the class of marked equivalences and $C$ the class of marked maps which are levelwise injective. We need to show that the classes $(W,C)$ meet the requirements of Theorem~\ref{existence}. We start by finding a set of morphisms which generates $C$ as a weakly saturated class.

Let $C_0$ to be the \textbf{set} containing all the morphisms
$$ \left[|\partial \Del^k| \otimes \left(\Del^n\right)^{\flat}\right] \coprod_{|\partial \Del^k| \otimes \left(\partial \Del^n\right)^{\flat}} \left[|\Del^k| \otimes \left(\partial \Del^n\right)^{\flat}\right] \hrar |\Del^k| \otimes (\Del^n)^{\flat} $$
and all the morphisms
$$ \left[|\partial \Del^k| \otimes \left(\Del^1\right)^{\sharp} \right]
\coprod_{|\partial \Del^k| \otimes \left(\Del^1\right)^{\flat}}
\left[|\Del^k| \otimes \left(\Del^1\right)^{\flat}\right] \hrar |\Del^k| \otimes \left(\Del^1\right)^{\sharp} .$$
It is not hard to check that $C$ is exactly the weakly saturated class generated from this set (these are standard arguments).

We will now show that $(W,C)$ satisfy the assumptions $2$ and $3$ of Theorem~\ref{existence}. Consider the category $\Set$ with its trivial model structure (i.e., the weak equivalences are the isomorphisms and all maps are fibrations and cofibrations). We endow $\Top^{\Del^{\op}_s} \times \Set$ with the \textbf{product model structure} (i.e., weak equivalences, fibrations and cofibrations are defined coordinate-wise, where on the left we use the Reedy model structure). Let $W',C'$ be the classes of weak equivalences and cofibrations in $\Top^{\Del^{\op}_s} \times \Set$ respectively.

Since both $\Top^{\Del^{\op}_s}$ and $\Set$ are left proper combinatorial model categories it follows that $\Top^{\Del^{\op}_s} \times \Set$ is a left proper combinatorial model category. This means that $W'$ is stable under pushouts along $C'$ and that $W'$ is perfect (this is part of Smith's theory of combinatorial model categories, cited for example in~\cite{higher-topos} A.2.6.6).

Now let $F:\Top_+^{\Del^{\op}_s} \lrar \Top^{\Del^{\op}_s} \times \Set$ be the functor given by $ F(X,A) = (X,\ovl{A}) $. Then it is clear that $F$ preserves colimits. Since
$ W = F^{-1}(W') $
and
$ C = F^{-1}(C') $
we get that $W$ is stable under pushouts along $C$ and that $W$ is perfect (see~\cite{higher-topos} A.2.6.12). It is then left to check the last assumption of Theorem~\ref{existence}.

Let $f: (X,A) \lrar (Y,B)$ be a morphism which has the right lifting property with respect to all maps in $C_0$. Since $C_0$ contains all maps of the form
$$ \left[|\partial \Del^k| \otimes \left(\Del^n\right)^{\flat} \right] \coprod_{|\partial \Del^k| \otimes \partial \left(\Del^n\right)^{\flat}} \left[|\Del^k| \otimes \left(\partial \Del^n\right)^{\flat}\right] \hrar  |\Del^k| \otimes (\Del^n)^{\flat} $$
it follows that $f$ is a levelwise equivalence. It is left to show that $f$ induces an isomorphism
$ \ovl{A} \lrar \ovl{B} $
Note that since $f$ is a levelwise equivalence it induces an isomorphism $\pi_0(X_1) \lrar \pi_0(Y_1)$ and so the map $\ovl{A} \lrar \ovl{B}$ is injective. The fact that it is surjective follows from having the right lifting property with respect to
$ \left(\Del^1\right)^{\flat} \hrar \left(\Del^1\right)^{\sharp} $
which is one of the maps in $C_0$. This completes the proof of Theorem~\ref{marked-model}.
\end{proof}

\begin{define}\label{marked-notation}
We will use the terms \textbf{marked fibrations}  and \textbf{marked cofibrations} to denote fibrations and cofibrations in the marked model structure. We will use the term \textbf{marked-fibrant semi-simplicial spaces} to denote fibrant objects in the marked model structure.
\end{define}

\begin{rem}\label{adjoint}
The forgetful functor $(X,A) \mapsto X$ from $\Top^{\Del^{\op}_s}_+$ to $\Top^{\Del^{\op}_s}$ is both a left and a right Quillen functor. As mentioned above, it has a right adjoint $X \mapsto X^{\sharp}$ and a left adjoint $X \mapsto X^{\flat}$. Furthermore it is easy to verify that both the forgetful functor and its left adjoint preserve cofibrations and weak equivalences.
\end{rem}

\begin{lem}\label{marked-fibrant}
A marked semi-simplicial space $(X,A)$ is \textbf{marked-fibrant} if and only if
\begin{enumerate}
\item
$X$ is Reedy fibrant.
\item
$A$ is a union of connected components of $X$.
\end{enumerate}
\end{lem}
\begin{proof}

Let $(X,A)$ be a marked-fibrant object. From remark~\ref{adjoint} we see that $X$ is Reedy fibrant. Now consider the maps
$$ \left[|\Lam_i^k| \otimes \left(\Del^1\right)^{\sharp}\right]
\coprod_{|\Lam_i^k| \otimes \left(\Del^1\right)^{\flat} }
\left[|\Del^k| \otimes \left(\Del^1\right)^{\flat} \right] \hrar |\Del^k| \otimes \left(\Del^1\right)^{\sharp} $$
for $k \geq 1$ and $0 \leq i \leq k$. By definition we see that these maps are trivial marked cofibrations. Since $(X,A)$ is Reedy fibrant it satisfies the right lifting property with respect to such maps, which in turn means that the inclusion $A \hrar X_1$ satisfies the right lifting property with respect to the inclusion of spaces $|\Lam_i^k| \hrar |\Del^k|$ for $k \geq 1$. This means that the inclusion $A \hrar X_1$ is Kan fibration and hence a union of components of $X_1$.

In the other direction assume that $X$ is Reedy fibrant and $A \subseteq X_1$ is a union of components. Consider an extension problem
$$ \xymatrix{
(Y,B) \ar^{f}[r]\ar[d] & (X,A) \\
(Z,C) & \\
}$$
such that $(Y,B) \hrar (Z,C)$ is a trivial marked cofibration. In this case $Y \hrar Z$ will be a trivial Reedy cofibration and so there will exist an extension $\ovl{f}:Z \lrar X$ in the category of semi-simplicial spaces. We claim that $\ovl{f}$ will necessarily send $C$ to $A$. In fact, let $W \subseteq Z_1$ be a connected component which meets $C$. Since $(Y,B) \hrar (Z,C)$ is a marked equivalences it follows that $W$ also meets the image of $B$. Since $A$ is a union of components of $X_1$ we get $\ovl{f}$ sends all of $W$ to $A$. This means that $\ovl{f}$ sends $C$ to $A$ and we are done.

\end{proof}

\begin{cor}\label{marked-fibrant-we}
A map $f: X \lrar Y$ between marked-fibrant semi-simplicial spaces is a marked equivalence if and only if it is a levelwise equivalence which induces a weak equivalence on the corresponding spaces of marked edges.
\end{cor}

We shall now show that $\Top^{\Del^{\op}_s}$ can be endowed with a structure of symmetric monoidal simplicial model category. Let $(X,A),(Y,B)$ be two marked semi-simplicial spaces. According to Remark~\ref{concrete} one has
$$ \left(X \otimes Y\right)_1 = \left(X_1 \times Y_0\right) \coprod \left(X_0 \times Y_1\right) \coprod \left(X_1 \times Y_1\right) .$$
We will extend the monoidal product $\otimes$ to marked semi-simplicial spaces by defining $(X,A) \otimes (Y,B)$ to be the marked semi-simplicial space $(X \otimes Y,C)$ where the marking $C$ is given by
$$ C = \left(A \times Y_0\right) \coprod \left(X_0 \times B\right) \coprod \left(A \times B\right) \subseteq \left(X \otimes Y\right)_1 .$$

The product $\otimes$ on $\Top^{\Del^{\op}_s}_+$ is again closed and the corresponding internal mapping object is defined as follows:
\begin{define}
Let $X, Y$ be two marked semi-simplicial spaces. The \textbf{marked mapping object} from $X$ to $Y$ is the marked semi-simplicial space $\left(Y^X,H\right)$ given by
$$ \left(Y^X\right)_n = \Map^+\left(X \times \left(\Del^n\right)^{\flat},Y\right) $$
where the marking $H$ is given by
$$ H = \Map^+\left(X \times \left(\Del^1\right)^{\sharp},Y\right) \subseteq \Map^+\left(X \times \left(\Del^1\right)^{\flat},Y\right) = \left(Y^X\right)_1 .$$
\end{define}

\begin{lem}
The marked model structure on $\Top^{\Del^{\op}_s}_+$ is \textbf{compatible} with the symmetric monoidal structure $\otimes$ (see Definition~\ref{d:compatible}).
\end{lem}
\begin{proof}
Since the Reedy model structure on $\Top^{\Del^{\op}_s}$ is compatible with the unmarked version of $\otimes$ we only need to verify the following: if $X' \lrar X$ is a marked cofibration and $Y' \lrar Y$ is a trivial marked cofibration then the map
$$ h:\left(X \otimes Y'\right)_1 \coprod_{(X' \otimes Y')_1} \left(X' \otimes Y\right)_1 \lrar (X \otimes Y)_1 $$
induces an isomorphism on the set of \textbf{marked connected components}. Since this map is already a weak equivalence it is enough to check that it is surjective on marked components. But this is a direct consequence of the fact that the map $Y'_1 \lrar Y_1$ induces an isomorphism on the set of marked connected components.
\end{proof}

Now, it is not hard to verify that the adjunction
$$
\xymatrix{
\Top^{\Del^{\op}_s} \ar@<0.5ex>[r]^{(\bullet)^{\flat}} &
\Top^{\Del^{\op}_s}_+   \ar@<0.5ex>[l]^{\ovl{\bullet}} \\
}
$$
is strongly monoidal (see Definition~\ref{d:quillen-monoidal}). In particular, the model category $\Top^{\Del^{\op}_s}_+$ inherits the simplicial structure of $\Top^{\Del^{\op}_s}$. To conclude $\Top^{\Del^{\op}_s}_+$ is a symmetric monoidal simplicial model category.

We finish this subsection with the following definition which we frame for future use:
\begin{define}\label{sub-groupoid}
Let $W$ be a marked semi-simplicial space with marking $M \subseteq W_1$. We will denote by $\wtl{W} \subseteq W$ the marked semi-simplicial space such that
$$ \wtl{W}_n = \left\{\sig \in W_n | f^*\sig \in M, \forall f:[1] \lrar [n]\right\} .$$
In particular, all the edges of $\wtl{W}$ are marked.
\end{define}

\subsection{ The marked right Kan extension }\label{sss-marked-right-kan}

The Reedy model structures on simplicial and semi-simplcial spaces can be related via a Quillen adjunction
$$ \xymatrix{
\Top^{\Del^{\op}} \ar@<0.5ex>[r]^{\F} & \Top^{\Del^{\op}_s} \ar@<0.5ex>[l]^{\RK} \\
}$$
where $\F$ is the forgetful functor and $\RK$ is the right Kan extension. The purpose of this section is to construct an analogous Quillen adjunction between $\Top^{\Del^{\op}}$ and the marked model structure on $\Top^{\Del^{\op}}_+$.

We begin by defining the \textbf{marked forgetful functor}
$$ \F^+: \Top^{\Del^{\op}} \lrar \Top^{\Del^{\op}_s}_+ $$
as follows: given a simplicial space $X$ we will define $\F^+(X)$ to be the marked semi-simplicial space $\left(\ovl{X},D\right)$ where $\ovl{X}$ is the underlying semi-simplicial space of $X$ and $D \subseteq \ovl{X}_1 = X_1$ is the subspace of \textbf{degenerate $1$-simplices}, i.e., the image of
$ s_0: X_0 \lrar X_1 $. The functor $\F^+$ admits a right adjoint
$$ \RK^+: \Top^{\Del^{\op}_s}_+ \lrar \Top^{\Del^{\op}} $$
which we shall now describe. For this we will need the following definition:

\begin{define}
Let $f: [m] \lrar [n]$ be a map in $\Del$. We will say that an edge $e \in (\Del^m)_1$ is \textbf{$f$-degenerate} if $f$ maps both its vertices to the same element of $[n]$. We will denote by $(\Del^m)^{f}$ the marked semi-simplicial space $(\Del^m,A_f)$ where $\Del^m$ is considered as a semi-simplicial space which is levelwise discrete and $A_f \subseteq (\Del^m)_1$ is the set of $f$-degenerate edges. Now given a marked semi-simplicial space $(X,A)$ we define
$$ X^{f}_m = \Map^{+}\left((\Del^m)^{f},(X,A)\right) .$$
Note that we have a natural inclusion
$ X^{f}_m \subseteq X_m $.
\end{define}

We will now construct the functor
$ \RK^+ $
as follows. For each $[n] \in \Del$ consider the fiber product category
$$ \C_n = \Del^{op}_s \times_{\Del^{op}} \Del^{op}_{[n]/} .$$
The objects of $\C_n$ can be identified with maps $f:[m] \lrar [n]$ in $\Del$ and a morphism from $f:[m] \lrar [n]$ to $g:[k] \lrar [n]$ in $\C_n$ can then be described as a commutative triangle
$$ \xymatrix{
[k] \ar_{g}[dr]\ar^{h}[rr] & & [m] \ar^{f}[dl] \\
& [n]  &\\
}$$
such that $h$ is \textbf{injective}. Now let $(X,A)$ be a marked semi-simplicial space and let
$ \mcal{G}_n: \C_n \lrar \Top $
be the functor which associates to each $f:[m] \lrar [n]$ the space
$ \mcal{G}_n(f) = X^{f}_m $.
Note that for each map $[n] \lrar [n']$ in $\Del$ one has a functor
$ \F_n: \C_n \lrar \C_n' $
and a natural transformation
$ \F_n^*\G_n' \lrar \G_n $.
We can then define $\RK^+(X,A)$ by setting
$$ \RK^+(X,A)_n = \lim_{\C_n} \mcal{G}_n $$
which is functorial in $[n] \in \Del$.

\begin{rem}\label{homotopy-limit}
The category $\C_n$ carries a Reedy structure which is induced from that of $\Del_s$. If $(X,A)$ is marked-fibrant then the functor $f \mapsto X^f_m$ will be a Reedy fibrant functor from $\C_n$ to $\Top$. This means that in this case the limit above will coincide with the respective \textbf{homotopy limit}.
\end{rem}

\begin{rem}
One has natural maps
$$ \RK^+(X,A)_n = \lim_{\C_n} X^{f}_m \lrar \lim_{\C_n} X_m = \RK(X)_n $$
which assemble together to form a natural transformation
$$ \RK^+(X,A) \lrar \RK(X) .$$
From Lemma~\ref{marked-fibrant} we see that when $(X,A)$ is marked-fibrant the map above identifies $\RK^+(X,A)_n$ with a union of connected components of $\RK(X)_n$.
\end{rem}

The pair $\F^+,\RK^+$ carries a compatible pair of unit and counit transformations exhibiting it as an adjunction
$$ \xymatrix{
\Top^{\Del^{\op}} \ar@<0.5ex>[r]^{\F^+} & \Top_+^{\Del^{\op}_s} \ar@<0.5ex>[l]^{\RK^+} \\
} .$$
This adjunction is easily seen to be a Quillen adjunction: since any Reedy cofibration in $\Top^{\Del^{\op}}$ is a levelwise injection it follows that $\F^+$ preserves cofibrations. Furthermore, it is not hard to check that $\F^+$ maps levelwise equivalences to marked equivalences, and hence trivial cofibrations to trivial marked cofibrations.

Now the forgetful functor
$ \F: \Top^{\Del^{\op}} \lrar \Top^{\Del^{\op}_s} $
factors through $\F^+$. This means that the Quillen adjunction
$$ \xymatrix{
\Top^{\Del^{\op}} \ar@<0.5ex>[r]^{\F} & \Top^{\Del^{\op}_s} \ar@<0.5ex>[l]^{\RK} \\
}$$
factors through $\Top^{\Del^{\op}_s}_+$ as the composition
$$ \xymatrix{
\Top^{\Del^{\op}}     \ar@<0.5ex>[r]^{\F^+} &
\Top^{\Del^{\op}_s}_+ \ar@<0.5ex>[l]^{\RK^+} \ar@<0.5ex>[r]^{\ovl{\bullet}} &
\Top^{\Del^{\op}_s}   \ar@<0.5ex>[l]^{(\bullet)^{\sharp}} \\
&  & \\
}$$
where $\ovl{\bullet}$ denotes the forgetful functor $(X,A) \mapsto X$.

\section{ Marked semiSegal spaces }\label{ss-marked-semiSegal}

In his paper~\cite{rez}, Rezk constructs a localization of the Reedy model structure on $\Top^{\Del^{\op}}$ in which the new fibrant objects are the Segal spaces. The purpose of this section is to construct a similar working environment for a suitable notion of a \textbf{marked semiSegal spaces}. Let us introduce the basic definitions:

\begin{define}\label{marked-segal}
Let $(W,M) \in \Top^{\Del^{\op}_s}_+$ be a \textbf{marked-fibrant} semi-simplicial space. We will say that $(W,M)$ is a \textbf{marked semiSegal space} if the following conditions are satisfied:
\begin{enumerate}
\item
$W$ is a semiSegal space.
\item
Every marked edge of $(W,M)$ is invertible, i.e., $M \subseteq W^{\inv}_1$.
\item
$M$ is closed under $2$-out-of-$3$, i.e., if there exists a triangle $\sig \in W_2$ with two marked edges then the third edge is marked as well.
\end{enumerate}
\end{define}

\begin{example}
Let $X$ be a \textbf{semiSegal space}. We will denote by $X^{\natural}$ the marked semi-simplicial space whose underlying semi-simplicial space is $X$ and whose marking is given by $X^{\inv}_1$. Then it is easy to verify that $X^{\natural}$ is a marked semiSegal space.
\end{example}

\begin{example}
Let $X$ be a \textbf{semiSegal space}. Then $X^{\flat}$ is a marked semiSegal space.
\end{example}

\begin{example}\label{e:cosk^+_0}
Let $Z$ be a Kan simplicial set. The $0$'th coskeleton functor (see Example~\ref{e:cosk_0}) has a marked analogue $\cosk^+_0: \Top \lrar \Top^{\Del^{\op}_s}_+$ which is given by $\cosk^+_0(Z) = \left(\cosk_0(Z)\right)^{\sharp}$ (this marked analogue is again the right adjoint to the functor $X \mapsto X_0$). Since $\cosk_0(Z)$ is a semiSegal space in which all edges are invertible we get that $\cosk^+_0(Z)$ is a marked semiSegal space.
\end{example}

We consider a marked semiSegal space as encoding a non-unital $\infty$-category in which a certain (suitably closed) subspace of the invertible morphisms has been marked. Given a marked semiSegal space $(W,M)$ we will denote by
$$ \Map^+_{W}(x,y) \subseteq \Map_W(x,y) $$
the subspace of marked edges from $x$ to $y$, i.e. the fiber of the map
$$ M \lrar X_0 \times X_0 $$
over $(x,y)$.

\begin{define}
We will say that $W$ is a \textbf{marked semiKan space} if $W$ is a marked semiSegal space in which all edges are marked. Note that in this case the underlying semi-simplicial space of $W$ is a semiKan space.
\end{define}

\begin{define}\label{d:quasi-unital}
We will say that a marked semiSegal space $(W,M)$ is \textbf{quasi-unital} if the following additional conditions are satisfied:
\begin{enumerate}
\item
The restricted fibration map $d_0: M \lrar X_0$ is surjective (i.e., every object admits a marked morphism out of it).
\item
Every invertible edge in $W$ is marked.
\end{enumerate}
We will denote by $\QsS \subseteq \Seg_s^{\fib}$ the full topological subcategory spanned by quasi-unital marked semiSegal spaces.
\end{define}

As explained in \S\ref{ss:qu-to-inv} we can consider a quasi-unital marked semiSegal space as encoding a quasi-unital $\infty$-category in which the invertible morphisms have been marked. As a result (see Proposition~\ref{unital-if-marked}) maps in $\QsS$ are the "unital" ones - they are maps of the underlying semiSegal spaces which send quasi-units to quasi-units.

As in the Segal space case, the topological category $\QsS$ is a good starting point for understanding the homotopy theory of quasi-unital $\infty$-categories. However, it is still not the correct model for it. The reason for this is that equivalences in $\QsS$ are too strict. The correct notion will be obtained after suitably localizing $\QsS$ with respect to a natural notion of Dwyer-Kan equivalences (see \S\S~\ref{ss-fully-faithful}).

It will be useful to describe the property of being quasi-unital in terms of a suitable right lifting property. Consider the following three maps of marked semi-simplicial sets (considered as levelwise discrete semi-simplicial spaces):
\begin{enumerate}
\item
The inclusion $C_0:\Del^{\{0\}} \hrar (\Del^1)^{\sharp}$.
\item
The inclusion $C_1:\Del^{\{1\}} \hrar (\Del^1)^{\sharp}$.
\item
The inclusion $C_2:\left(\Del^3,A\right) \hrar \left(\Del^3\right)^{\sharp}$ where $A = \left\{\Del^{\{0,2\}},\Del^{\{1,3\}}\right\}$.
\end{enumerate}
We then have the following simple observation:
\begin{lem}\label{qu-rlp}
Let $(X,M)$ be a marked semiSegal space and let $T$ be the terminal marked semi-simplicial space (i.e. $T_n$ is a point for every $n$ and the edge in $T_1$ is marked). Then $X$ is quasi-unital if and only if the terminal map $X \lrar T$ has the right lifting property with respect to $C_0,C_1,C_2$ above.
\end{lem}
\begin{proof}
Condition $(1)$ of Definition~\ref{d:quasi-unital} is clearly equivalent to having the right lifting property with respect to $C_0,C_1$. The right lifting property with respect to $C_2$ is equivalent to the space of marked edges satisfying the $2$-out-of-$6$ rule. When condition $(1)$ is satisfied this is equivalent to the having all the equivalences marked.
\end{proof}

The remainder of this section is organized as follows. To begin, we will localize the marked model structure on $\Top^{\Del^{\op}_s}_+$ so that marked semiSegal spaces will coincide with the new fibrant objects. The construction of this localization as well as the verification of its compatibility with the monoidal structure will be taken up in \S\S~\ref{ss:marked-segal-model}. In particular, we will obtain a notion of internal mapping objects for marked semiSegal spaces.

In \S\S~\ref{ss-fully-faithful} we will study the notions of fully-faithful maps and Dwyer-Kan equivalences (DK-equivalences for short) between marked semiSegal spaces. Our suggested model for the homotopy theory of quasi-unital $\infty$-categories is the localization of $\QsS$ with respect to DK-equivalences. In \S\S~\ref{ss-Q-anodyne} we will attempt to study the notion of quasi-unitality from a model categorical point of view. This will allow us to establish some useful results which will be exploited in the final section, e.g., we will prove that the full subcategory $\QsS$ is closed under taking mapping objects. This can be considered as a step towards the main theorem as well as it essentially says that when quasi-units exist then they can be chosen coherently over arbitrary families.

\subsection{ The marked semiSegal model structure }\label{ss:marked-segal-model}

The purpose of this subsection is to show that one can identify the full subcategory of marked semiSegal spaces with the subcategory of fibrant objects in a suitable localization of the marked model category $\Top^{\Del^{\op}_s}_+$. In order to do this we will need to identify a set of maps such that the condition of being a marked semiSegal space can be expressed as a locality condition with respect to these maps. To describe this conveniently we will need a bit of terminology.

We will use the phrase \textbf{marked horn inclusion} to describe an inclusion of marked semi-simplicial sets of the form
$$ (\Lam^n_i,A) \subseteq (\Del^n,B) $$
such that $A = B \cap (\Lam^n_i)_1$. We will be interested in the following kind of marked horn inclusions:
\begin{define}\label{admissible}
We will say that a marked horn inclusion
$$ (\Lam^n_i,A) \subseteq (\Del^n,B) $$
is \textbf{admissible} if $B = A$, $n \geq 2$ and in addition one of the following (mutually exclusive) conditions is satisfied:
\begin{enumerate}
\item
$0 < i < n$ and $A = \emptyset$.
\item
$i=0$ and $A = \left\{\Del^{\{0,1\}}\right\}$.
\item
$i=n$ and $A = \left\{\Del^{\{n-1,n\}}\right\}$.
\end{enumerate}
\end{define}

Our purpose is to show that conditions $(1)$ and $(2)$ of Definition~\ref{marked-segal} can be formulated in terms of locality with respect to admissible marked horn inclusions. We begin with the following technical lemma:
\begin{lem}\label{induction-lem-1}
Let $n \geq 2$ and $0 \leq i,j \leq n$. Define $A = \{0,...,j-1\} \cup \{i\}$ and $B = \{j,...,n\} \cup \{i\}$. Let
$$ X = \Del^A \coprod_{\Del^{\{i\}}} \Del^B \hrar \Del^n $$
be the corresponding subcomplex of $\Del^n$ and let $M \subseteq X_1$ be the set of all edges of the form $\Del^{\{i,x\}} \subseteq X$ for $x \in B, x \geq i$ and all edges of the form $\Del^{\{x,i\}}$ for $x \in A, x \leq i$. Then the marked semi-simplicial set $\left(\Lam^n_i,M\right)$ can be obtained from $X$ be performing pushouts along admissible marked horn inclusions.
\end{lem}
\begin{proof}
We will say that a simplex $\Del^J \subseteq \Del^n$ for $|J| \geq 3$ is \textbf{spread} if $J$ contains $i$ and $J$ has a non-empty intersection with both $A \bksl \{i\}$ and $B \bksl \{i\}$. Now define
$$ X = X_1 \subseteq X_2 \subseteq ... \subseteq X_{n-1} $$
inductively by letting $X_k$ be the union of $X_{k-1}$ and all the spread simplices of dimension $k$. Note that if $\Del^J \subseteq \Del^n$ is spread of dimension $|J| = k$ then $\Del^J \cap X_j$ is the horn of $\Del^J$ which contains all the $(k-1)$-faces except the face opposite the vertex $i \in J$. Furthermore, if $i$ is a maximal element of $J$ then the second biggest element of $J$ is in $B$ and so the last edge of $\Del^J$ is marked. Similarly if $i$ is a minimal element of $J$ then the second smallest element of $J$ is in $A$ and so the first edge of $\Del^J$ is marked. In either case the addition of $\Del^J$ can be preformed by a pushout along an admissible marked horn inclusion. To finish the proof note that $X_{n-1} = \left(\Lam^n_i,M\right)$.
\end{proof}

Lemma~\ref{induction-lem-1} has the following corollary:
\begin{cor}\label{admissible-marked-horn-inclusion}
Let $(X,M)$ be a fibrant marked semi-simplicial space. Then
\begin{enumerate}
\item
$X$ satisfies condition $(1)$ of Definition~\ref{marked-segal} (i.e. the Segal condition) if and only if $X$ is local with respect to inner admissible horn inclusions.
\item
Assume that $X$ satisfies the Segal condition. Then $X$ satisfies condition $(2)$ of Definition~\ref{marked-segal} if and only if $X$ is local with respect to non-inner admissible horn inclusions.
\end{enumerate}
\end{cor}
\begin{proof}
Part $(1)$ follows from Lemma~\ref{induction-lem-1} with $0 < i=j < n$ (in which case no marking is involved) together with a simple inductive argument. Now assume that $X$ satisfies the Segal condition. If $X$ is local with respect to admissible horn inclusions then in particular $X$ is local with respect to the inclusions
$$ \left(\Lam^2_0,\left\{\Del^{\{0,1\}}\right\}\right) \hrar \left(\Del^2, \left\{\Del^{\{0,1\}}\right\}\right) $$
and
$$ \left(\Lam^2_2,\left\{\Del^{\{1,2\}}\right\}\right) \hrar \left(\Del^2, \left\{\Del^{\{1,2\}}\right\}\right) $$
and so by Remark~\ref{inv-concrete} every marked edge in $X$ is invertible.

Now assume that every marked edge is invertible. Then $X$ is local with respect to admissible horn inclusions of dimension $2$. Assume by induction that $X$ is local with respect to all admissible horn inclusions of dimension $k \geq 2$. Applying Lemma~\ref{induction-lem-1} for the case $j=k,i=n=k+1$ we get from the induction hypothesis that $X$ is local with respect to the inclusion
$$ \Del^A \coprod_{\Del^{\{i\}}} \Del^B \hrar \left(\Lam^n_n,\{\Del^{\{n-1,n\}}\}\right) \subseteq \Del^n $$
and so
$$ \Map_+\left(\left(\Lam^n_n,\{\Del^{\{n-1,n\}}\}\right),X\right) \cong X_{|A|} \times_{X_0} M \cong X_1 \times_{X_0} ... \times_{X_0} X_1 \times_{X_0} M $$
where in the fiber product $X_1 \times_{X_0} M$ at the end of the right hand side both maps $X \lrar X_0$ and $M \lrar X_0$ are induced by $d_1$. Using locality with respect to admissible horn inclusion of dimension $2$ one sees that this fiber product is equivalent to the similar fiber product $X_1 \times_{X_0} M$ now involving the map $d_0: M \lrar X_0$ instead of $d_1$. The desired result now follows from the Segal condition.

\end{proof}

Now let $W \in \Top^{\Del^{\op}_s}_+$ be a marked-fibrant object. In light of Corollary~\ref{admissible-marked-horn-inclusion} we see that $W$ will be a semiSegal space if and only if $W$ is local with respect to the set $\BS$ defined as follows:
\begin{define}\label{local-generators}
Let $\BS$ be the set which contains:
\begin{enumerate}
\item
All admissible marked horn inclusions.
\item
All the maps of the form
$$ \left(\Del^2,A\right) \hrar \left(\Del^2\right)^{\sharp} $$
where $A \subseteq \left(\Del^2\right)_1$ a set of size $2$.
\end{enumerate}
\end{define}

We are now in a position to define our desired model category. Since the marked model structure is combinatorial and left proper the left Bousfield localization of $\Top^{\Del^{\op}_s}_+$ with respect to $\BS$ exists. In particular, there exists a (combinatorial, left proper) model category $\Seg_s$ whose underlying category is $\Top^{\Del^{\op}_s}_+$ such that
\begin{enumerate}
\item
Weak equivalences in $\Seg_s$ are maps $f: X \lrar Y$ such that for every marked semiSegal space $W$ the induced map
$$ \Map^+(Y, W) \lrar \Map(X,W) $$
is a weak equivalence.
\item
Cofibrations in $\Seg_s$ are the cofibrations of the marked model structure (i.e., levelwise injective maps).
\item
The fibrant objects in $\Seg_s$ are precisely the marked semiSegal spaces.
\end{enumerate}

\begin{define}
We will refer to $\Seg_s$ as the \textbf{marked semiSegal model structure}. We will denote by \textbf{MS-equivalences}, \textbf{MS-fibrations} and \textbf{MS-cofibrations} the weak equivalences, fibrations and cofibrations in $\Seg_s$ respectively (to avoid confusion compare to the terminology in Definition~\ref{marked-notation}). Note that the notions of an MS-cofibration and a marked cofibration coincide.
\end{define}

The following kind of trivial MS-cofibrations will be useful to note:
\begin{define}\label{triangle-remarking}
Let $X$ be a marked semi-simplicial space and $B \subseteq C \subseteq X_1$ two subspaces. We will say that the map
$$ (X,B) \hrar (X,C) $$
is a \textbf{triangle remarking} if $(X,C)$ can be obtained from $(X,B)$ by a sequence of pushouts along maps of the form
\begin{equation}\label{e:remarking}
\left[L \otimes \left(\Del^2,A\right)\right] \coprod_{K \otimes \left(\Del^2,A\right)} \left[K \otimes \left(\Del^2\right)^{\sharp}\right] \hrar L \otimes \left(\Del^2\right)^{\sharp}
\end{equation}
for $f:K \hrar L$ an inclusion of \textbf{spaces} and $|A| = 2$ (here we consider spaces as marked semi-simplicial spaces which are concentrated in degree $0$).
\end{define}

\begin{lem}
Any triangle remarking is a trivial MS-cofibration.
\end{lem}
\begin{proof}
Note that the claim is clearly true if $K = \emptyset$ and $L$ is discrete. Hence the claim is also true for maps of the form $K \hrar K \coprod L'$ where $L'$ is discrete. Orthogonally, if the inclusion $K \hrar L$ is surjective on connected components then the map ~\ref{e:remarking} is a marked equivalence and so in particular an MS-equivalence. The result now follows by factoring a general inclusion $K \hrar L$ as $K \hrar K \coprod L' \lrar L$ where $L'$ is discrete and the map $K \coprod L' \lrar L$ is surjective on connected components.
\end{proof}

This notion is exemplified in the following lemma:

\begin{lem}\label{horn-example}
For every $i=0,...,n$ the map
$$ \left(\Lam^n_i\right)^{\sharp} \lrar \left(\Del^n\right)^{\sharp} $$
is a trivial MS-cofibration.
\end{lem}
\begin{proof}
Let $M \subseteq (\Del^n)_1$ be the set of edges that are contained in $\Lam^n_i$. Then $\left(\Del^n,M\right)$ is obtained from $\left(\Lam^n_i\right)^{\sharp}$ by performing a pushout along an admissible marked horn inclusion. The desired result now follows from the fact that the map
$$ \left(\Del^n,M\right) \hrar \left(\Del^n\right)^{\sharp} $$
is a triangle remarking.
\end{proof}

\begin{cor}\label{max-groupoid-qu}
If $W$ is a marked semiSegal space then $\wtl{W}$ (see Definition~\ref{sub-groupoid}) is a marked semiSegal space as well.
\end{cor}
\begin{proof}
First of all it is clear that $\wtl{W}$ is marked-fibrant (see Lemma~\ref{marked-fibrant}). From Lemma~\ref{horn-example} it follows that $\wtl{W}$ is local with respect to all admissible marked horn inclusions and so by Proposition~\ref{admissible-marked-horn-inclusion} $\wtl{W}$ satisfies properties $(1)$ and $(2)$ of Definition~\ref{marked-segal}. Since clearly the marked edges in $\wtl{W}$ are closed under $2$-out-of-$3$ we get that $\wtl{W}$ is a marked semiSegal space.
\end{proof}

\begin{rem}\label{max-groupoid-2}
Since all the edges in $\wtl{W}$ are marked we see that $\wtl{W}$ is a \textbf{marked semiKan space}.
\end{rem}

\begin{rem}\label{max-groupoid-3}
If $W$ is quasi-unital then $\wtl{W}$ is quasi-unital as well. Furthermore, in this case $\wtl{W}$ contains all invertible edges of $W$.
\end{rem}

Now recall that $\Top^{\Del^{\op}_s}_+$ is a symmetric monoidal model category with respect to the marked monoidal product $\otimes$. We would like to show that this monoidality survives the localization:

\begin{thm}\label{marked-monoidality}
The marked Segal model structure is compatible with the marked monoidal product $\otimes$. In particular, the localization Quillen adjunction
$$ \xymatrix{
\Top_+^{\Del_s^{\op}} \ar@<0.5ex>[r]^{\Id} & \Seg_s \ar@<0.5ex>[l]^{\Id} \\
}$$
is strongly monoidal and $\Seg_s$ inherits the simplicial structure of $\Top^{\Del_s^{\op}}$.
\end{thm}

\begin{proof}

Arguing as in~\cite{rez} Proposition 9.2, we see that it will be enough to show that for every marked semiSegal space $W$ the objects $W^{\left(\Del^n\right)^{\flat}}$ and $W^{\left(\Del^1\right)^{\sharp}}$ are marked semiSegal spaces for every $n \geq 0$. This, in turn, can be easily reduced to checking that for every $f: Y \lrar Z$ in $\BS$ the inclusions
$$ \left[ \left(\partial\Del^m\right)^{\flat} \otimes Z\right] \coprod_{\left(\partial\Del^m\right)^{\flat} \otimes Y} \left[\left(\Del^m\right)^{\flat} \otimes Y\right] \hrar \left(\Del^m\right)^{\flat} \otimes Z $$
and
$$ \left[\left(\Del^1\right)^{\flat} \otimes Z\right] \coprod_{\left(\Del^1\right)^{\flat} \otimes Y} \left[\left(\Del^1\right)^{\sharp} \otimes Y\right] \hrar \left(\Del^1\right)^{\sharp} \otimes Z $$
are trivial MS-cofibrations. Note that the case $m=0$ above is trivial so we can assume $m \geq 1$.

We begin by clearing up some trivial cases. Observe that for a pair of inclusions of the form
$$ f: (X,A) \hrar (Y,A) $$
$$ g: (Z,B) \hrar (Z,C) $$
such that $f_0: X_0 \lrar Y_0$ is \textbf{surjective} the resulting map
$$ \left[(Z,C) \otimes (X,A)\right] \coprod_{(Z,B) \otimes (X,A)} \left[(Z,B) \otimes (Y,A)\right] \lrar (Z,C) \otimes (Y,A) $$
is in fact an \textbf{isomorphism} (and in particular a trivial MS-cofibration). Considering the various types of maps in $\BS$ one sees that the only cases which are not covered by the above argument are the following:

\begin{enumerate}
\item
The maps of the form
$$ \left[\left(\Del^1\right)^{\sharp} \otimes \left(\Del^2,A\right)\right] \coprod_{\left(\Del^1\right)^{\flat} \otimes \left(\Del^2,A\right)} \left[\left(\Del^1\right)^{\flat} \otimes \left(\Del^2\right)^{\sharp}\right] \lrar \left(\Del^1\right)^{\sharp} \otimes \left(\Del^2\right)^{\sharp} = \left(\Del^1 \otimes \Del^2\right)^{\sharp} $$
where $|A| = 2$.
\item
The maps of the form
$$ \left(\partial \Del^m\right)^{\flat} \otimes \left(\Del^n,A\right)  \coprod_{\left(\partial \Del^m\right)^{\flat} \otimes \left(\Lam^n_l,A\right) }  \left(\Del^m\right)^{\flat} \otimes \left(\Lam^n_l,A\right) \hrar \left(\Del^m\right)^{\flat} \otimes \left(\Del^n,A\right) $$
where $\left(\Lam^n_l,A\right) \hrar \left(\Del^n,A\right)$ is an admissible marked horn inclusion.
\end{enumerate}

For case $(1)$, note that this map induces an isomorphism on the underlying semi-simplicial sets. Furthermore, the marking on the left hand side contains all edges except exactly \textbf{one} edge $e \in \left(\Del^1 \otimes \Del^2\right)_1$.

Note that each triangle in $\Del^1 \otimes \Del^2$ has three distinct edges. Furthermore every edge in $\Del^1 \otimes \Del^2$ lies on some triangle. Hence one can find a triangle which lies on $e$ such that its other two edges are not $e$. This means that there exists a pushout diagram of marked semi-simplicial sets of the form
$$ \xymatrix{
\left(\Del^2,A\right) \ar[d]\ar[r] & \left(\Del^2\right)^{\sharp} \ar[d] \\
\left[\left(\Del^1\right)^{\sharp} \otimes \left(\Del^2,A\right)\right] \coprod_{\left(\Del^1\right)^{\flat} \otimes \left(\Del^2,A\right)} \left[\left(\Del^1\right)^{\flat} \otimes \left(\Del^2\right)^{\sharp}\right] \ar[r] & \left(\Del^1 \otimes \Del^2\right)^{\sharp}
}$$
Since the upper horizontal row is a trivial MS-cofibration we get that the lower horizontal map is an MS-cofibration as well.

We shall now prove case $(2)$:

\begin{lem}\label{horn-pushouts-admissible}
Let $\left(\Lam^n_l,A\right) \hrar \left(\Del^n,A\right)$ be an admissible marked horn inclusion. Then the marked semi-simplicial set $\left(\Del^m\right)^{\flat} \otimes \left(\Del^n,A\right)$ can be obtained from the marked semi-simplicial set
$$ X =  \left(\partial \Del^m\right)^{\flat} \otimes \left(\Del^n,A\right)  \coprod_{\left(\partial \Del^m\right)^{\flat} \otimes \left(\Lam^n_l,A\right) }  \left(\Del^m\right)^{\flat} \otimes \left(\Lam^n_l,A\right) $$
by successively performing pushouts along admissible marked horn inclusions. In particular, the inclusion
$$ X \subseteq \left(\Del^m\right)^{\flat} \otimes \left(\Del^n,A\right) $$
is a trivial MS-cofibration.
\end{lem}

\begin{proof}

If $m=0$ then the claim is immediate, so we can assume $m > 0$. In this case the marking of $\left(\Del^m\right)^{\flat} \otimes \left(\Del^n,A\right)$ is the same as the marking of $X$, so that we don't need to worry about adding marked edges in the course of performing the desired pushouts. Note that we can harmlessly assume that $0 < l \leq n$ (as the case $0 \leq l < n$ follows from symmetry).

According to Remark~\ref{concrete} the $k$-simplices of $\Del^m \otimes \Del^n$ are in one-to-one correspondence with injective order preserving maps
$$ \sig = (f,g): [k] \lrar [m] \times [n] .$$
We will consider a $k$-simplex $\sig = (f,g)$ as above as an injective marked map
$$ \sig: (\Del^k,B) \hrar \left(\Del^m\right)^{\flat} \otimes \left(\Del^n,A\right) $$
where $B$ is defined to be $\left\{\Del^{\{k-1,k\}}\right\}$ be $\sig\left(\Del^{\{k-1,k\}}\right)$ is marked and $\emptyset$ otherwise.

We will say that a $k$-simplex of $\Del^m \otimes \Del^n$ is \textbf{full} if it is \textbf{not} contained in $X$. If we describe our $k$-simplex by a map $\sig = (f,g)$ as above this translates to the condition that $f$ is surjective and that the image of $g$ contains $\{0,...,n\} \bksl \{l\}$ (so that $g$ is either surjective or misses $l$). Our purpose is to add all the full simplices to $X$ in a way that involves only pushouts along admissible horn inclusions. For this we distinguish between two kinds of $k$-simplices of $\Del^m \otimes \Del^n$:
\begin{define}
Let
$$ \sig = (f,g): [k] \lrar [m] \times [n] $$
be a $k$-simplex of $\Del^m \otimes \Del^n$. We will say that $\sig$ is \textbf{special} if
\begin{enumerate}
\item
$\sig$ is full.
\item
$g^{-1}(l) \neq \emptyset$.
\item
$ f\left(\min g^{-1}(l)\right) = f\left(\max g^{-1}(l-1)\right)$.
\end{enumerate}
If $\sig$ is full but not special then we will say that $\sig$ is \textbf{regular}.
\end{define}

Now for $i=0,...,m+1$ let $X_i$ denote the union of $X$ and all \textbf{special} $(i+n-1)$-simplices of $\Del^m \otimes \Del^n$. We now claim the following:
\begin{enumerate}
\item
$X_0 = X$.
\item
For $i=0,...,m$ the semi-simplicial set $X_{i+1}$ is obtained from $X_i$ by a sequence of pushouts along admissible horn inclusions of dimension $i+n$.
\item
$X_{m+1} = \Del^m \otimes \Del^n$.
\end{enumerate}
The first claim just follows from the fact that there are no special simplices of dimension less than $n$. Now $X_{i+1}$ is the union of $X_i$ and all special $(i+n)$-simplices. Hence in order to prove the second claim we will need to find the right \textbf{order} in which to add these special $(i+n)$-simplices to $X_i$. We will do this by sorting them according to the following quantity:
\begin{define}
Let
$$ \sig = (f,g): [k] \lrar [m] \times [n] $$
be a \textbf{full} $k$-simplex of $\Del^m \otimes \Del^n$. We define the \textbf{index} of $\sig$ to be the quantity
$$ \ind(\sig) = k+1-n-|g^{-1}(l)| .$$
\end{define}
Note that for a general full simplex the index is a number between $0$ and $k+1-n$. By definition we see that for a \textbf{special} $k$-simplex the index is a number between $0$ and $k-n$. In particular, the index of a special $(i+n)$-simplex is a number between $0$ and $i$.

Now fix an $i = 0,...,m$ and for each $j=0,...,i+1$ define $X_{i,j}$ to be the union of $X_{i}$ and all special $(i+n)$-simplices $\sig$ whose index is strictly less than $j$. We obtain a filtration of the form
$$ X_{i} = X_{i,0} \subseteq X_{i,1} \subseteq ... \subseteq X_{i,i+1} = X_{i+1} .$$
We will show that if $\sig$ is a special $(i+n)$-simplex of index $j$ then the intersection
$$ \sig \cap X_{i,j} $$
is an admissible horn of $\sig$ (with respect to the marking induced from $\sig$). This means that $X_{i,j+1}$ can be obtained from $X_{i,j}$ by performing pushouts along admissible horn inclusions of dimension $m+i$, implying the second claim above. We start by noting that if $\tau = (f,g)$ is a \textbf{regular} $k$-simplex then $\tau$ is a face of the special $(k+1)$-simplex
$$ \sig = (f \circ s_{\max g^{-1}(l-1)},g \circ s_{\min g^{-1}(l)}) $$
where $s_r: [k+1] \lrar [k]$ is the degeneracy map hitting $r$ twice. Furthermore we see that $\ind(\sig) = \ind(\tau)$. This means that $X_{i,j}$ contains in particular all regular $(i+n-1)$-simplices whose index is $<j$. Since taking faces cannot increase the index we see that an $(i+n-1)$-simplex $\tau$ is contained in $X_{i,j}$ exactly when $\tau$ is not regular of index $\geq j$.

Now let $\sig = (f,g)$ be a \textbf{special} $(i+n)$-simplex of index $j$ and let $\tau$ be the $(i+n-1)$-face of $\sig$ which is opposed to the $v$'th vertex for $v = 0,...,i+n$. Then we see that $\tau$ will be regular of index $\geq j$ if and only if $v = \min g^{-1}(l)$, in which case $\ind(\tau) = \ind(\sig) = j$. Since $g$ is surjective we get that
$$ 0 < \min g^{-1}(l) \leq i+l \leq i+n $$
and so $X_{i,j} \cap \sig$ is a right horn of $\sig$ which is inner if $l < n$. In fact, the only case where this right horn inclusion is not inner is when $\min g^{-1}(l) = k$. By the definition of special we then have
$$ f\left(k\right) = f\left(k-1\right) $$
and so the $\{k-1,k\}$-edge of $\sig$ is mapped to a \textbf{marked} edge in $\left(\Del^m\right)^{\flat} \otimes \left(\Del^n,A\right)$. This means that indeed the addition of $\sig$ can be done by a pushout along an \textbf{admissible} horn inclusion.

It is left to prove the third claim, i.e., that $X_{m+1} = \Del^m \otimes \Del^n$. From the considerations above we see that $X_{i+1}$ contains all full $k$-simplices for $k < n+i$ (as well as all special $(n+i)$-simplices). Since all the full $(m+n)$-simplices are special we get that $X_{m+1}$ contains all full simplices of $\Del^m \otimes \Del^n$ of dimension up to $m+n$, yielding the desired result.

\end{proof}

This finishes the proof of Theorem~\ref{marked-monoidality}.
\end{proof}

\begin{cor}
Let $W$ be a marked semiSegal space and $X$ a marked semi-simplicial space. Then $W^X$ is a marked semiSegal space and $\wtl{W^X}$ is a marked semiKan space.
\end{cor}

\subsection{ Fully-faithful maps and Dwyer-Kan equivalences }\label{ss-fully-faithful}

The purpose of this section is to study the notion of \textbf{fully-faithful maps} and \textbf{Dwyer-Kan equivalences} in the setting of marked semi-simplicial spaces. We begin with the basic definition:

\begin{define}
Let $f: (W,M) \lrar (Z,N)$ be a map of semi-simplicial spaces. We will say that $f$ is \textbf{fully-faithful} if the squares
\begin{equation}
\begin{split}
\xymatrix{
W_n \ar[r]\ar[d] & Z_n \ar[d] \\
\left(W_0\right)^{n+1} \ar[r] & \left(Z_0\right)^{n+1} \\
}\;\;\;\;
\xymatrix{
M \ar[r]\ar[d] & N \ar[d] \\
W_0 \times W_0 \ar[r] & Z_0 \times Z_0 \\
}
\end{split}
\end{equation}
are homotopy Cartesian for every $n \geq 1$.

\end{define}

In this paper we will often encounter maps $f: W \lrar Z$ of marked semiSegal spaces which are simultaneously fully-faithful and a marked fibration. This case admits a particularly nice description. It will be convenient to employ the following terminology (which makes sense in any simplicial category):

\begin{define}\label{d:contractible}
Let $g: X \lrar Y$, $f: W \lrar Z$ be two maps in $\Top^{\Del^{\op}_s}_+$. We will say that $f$ has the \textbf{contractible} right lifting property with respect to $g$ if the map
\begin{equation}\label{e:contractible}
\Map_+(Y,W) \lrar \Map_+(Y,Z) \times_{\Map_+(X,Z)} \Map_+(X,W)
\end{equation}
is a trivial Kan fibration.
\end{define}

\begin{rem}\label{r:contractible}
Definition~\ref{d:contractible} is equivalent to saying that $p$ has the right lifting property with respect to the maps
$$ \left|\partial \Del^m\right| \otimes Y \coprod_{\left|\partial \Del^m\right| \otimes X }  \left|\Del^m\right| \otimes X \hrar \left|\Del^m\right| \otimes Y $$
for every $m \geq 0$.
\end{rem}

\begin{rem}\label{r:fib-fully-faithful}
When $f:W \lrar Z$ is a marked fibration and $g: X \lrar Y$ is a marked cofibration then the map~\ref{e:contractible} is a fibration. Furthermore, in this case the fiber product on the right hand side coincides with the homotopy fiber product. From this observation we see that a marked fibration $f: W \lrar Z$ is fully-faithful if and only if is satisfies the contractible right lifting property with respect to the maps
$$ \overbrace{\Del^0 \coprod ... \coprod \Del^0}^{n+1} \hrar \left(\Del^n\right)^{\flat} $$
for every $n \geq 1$ and the map
$$ \Del^0 \coprod \Del^0 \hrar \left(\Del^1\right)^{\sharp} .$$
\end{rem}

We are now ready to prove our main characterization theorem concerning fully-faithful marked fibrations:

\begin{prop}\label{lift-character}
Let $f: W \lrar Z$ be a map of semi-simplicial spaces. Then the following assertions are equivalent:
\begin{enumerate}
\item
$f$ is a fully-faithful marked fibration.
\item
$f$ is a marked fibration and satisfies the contractible right lifting property with respect to the maps
$ \left(\partial \Del^n\right)^{\flat} \hrar \left(\Del^n\right)^{\flat} $
for every $n \geq 1$ and the map
$ \partial \Del^1  \hrar \left(\Del^1\right)^{\sharp} $
\item
$f$ satisfies the right lifting property with respect to every marked cofibration $g: X \hrar Y$ such that $g_0: X_0 \hrar Y_0$ is a weak equivalence.
\item
$f$ satisfies the contractible right lifting property with respect to every marked cofibration $g: X \hrar Y$ such that $g_0: X_0 \hrar Y_0$ is a weak equivalence.
\item
$f$ is a fully-faithful MS-fibration.
\item
$f$ is an MS-fibration and satisfies the contractible right lifting property with respect to the maps
$ \partial\Del^1  \hrar \left(\Del^1\right)^{\flat} $
and
$ \partial \Del^1  \hrar \left(\Del^1\right)^{\sharp}  $
\end{enumerate}
\end{prop}
\begin{proof}
\
\begin{enumerate}
\item [$(1) \Rightarrow (2)$]
Invoking Remark~\ref{r:fib-fully-faithful} we note that the semi-simplicial set $\left(\partial \Del^n\right)^{\flat}$ can be obtained from $\overbrace{\Del^0 \coprod ... \coprod \Del^0}^{n+1}$ by successively performing pushouts along maps of the form
$$ \left(\partial \Del^k\right)^{\flat} \hrar \left(\Del^k\right)^{\flat} $$
for $k < n$. Hence the claim follows by induction on $n$.
\item[$(2) \Rightarrow (3)$]
Assume $f$ satisfies $(2)$ and let $g: X \hrar Y$ be a marked cofibration such that $g_0: X_0 \lrar Y_0$ is a weak equivalence. Then one can factor $g$ as
$$ X \x{g'}{\lrar} X' \x{g''}{\lrar} Y $$
such that $g'$ is a trivial marked cofibration and $g''$ induces an isomorphism $g''_0: X_0' \lrar Y_0$. Since $f$ is a marked fibration it will suffice to show that $f$ satisfies the right lifting property with respect to $g''$. But this follows by Remark~\ref{r:contractible} from the fact that $g''$ can be written as a (transfinite) composition of pushouts along maps of the form
$$ \left[\left|\Del^m\right| \otimes \left(\partial \Del^n\right)^{\flat}\right] \coprod_{\left|\partial\Del^m\right| \otimes \left(\partial\Del^n\right)^{\flat}}\left[\left|\partial\Del^m\right| \otimes \left(\Del^n\right)^{\flat}\right]  \hrar \left|\Del^m\right| \otimes \left(\Del^n\right)^{\flat} $$
and
$$ \left[\left|\Del^m\right| \otimes \partial \Del^1\right] \coprod_{\left|\partial\Del^m\right| \otimes \partial\Del^1}\left[\left|\partial\Del^m\right| \otimes \left(\Del^1\right)^{\sharp}\right]  \hrar \left|\Del^m\right| \otimes \left(\Del^1\right)^{\sharp} $$
for $m \geq 0$.
\item [$(3) \Rightarrow (4)$]
This implication follows from the fact that the family of cofibrations $g: X \lrar Y$ such that $g_0: X_0 \lrar Y_0$ is a weak equivalence is stable under replacing $g$ with
$$ \left|\partial \Del^m\right| \otimes Y \coprod_{\left|\partial \Del^m\right| \otimes X }  \left|\Del^m\right| \otimes X \hrar \left|\Del^m\right| \otimes Y .$$
\item [$(4) \Rightarrow (5)$]
Assume that $f$ satisfies $(4)$. Then it is straightforward (using Remark~\ref{r:fib-fully-faithful}) to deduce that $f$ is fully-faithful. To show that $f$ is an MS-fibration it will be enough to show that if $g: X \lrar Y$ is a trivial MS-cofibration then $g_0: X_0 \hrar Y_0$ is a weak equivalence. This in turn follows from the fact that if $Z$ is a Kan simplicial set then $\cosk^+_0(Z)$ is a marked semiSegal space (see Example~\ref{e:cosk^+_0}) which means that $\Map_{\Top}(Y_0,Z) \lrar \Map_{\Top}(X_0,Z)$ is a weak equivalence for every $Z$.
\item [$(5) \Rightarrow (6)$]
Follows from Remark~\ref{r:fib-fully-faithful}.
\item [$(6) \Rightarrow (1)$]
First note that if $f$ satisfies $(6)$ then it has the contractible right lifting property with respect to the map
$$ \overbrace{\Del^0 \coprod ... \coprod \Del^0}^{n+1} \hrar \left(\Sp^n\right)^{\flat} $$
where $\Sp^n \subseteq \Del^n$ is the $n$-spine (see \S\S~\ref{ss-semi-simplicial}). The desired result now follows from the fact that the inclusion
$$ \left(\Sp^n\right)^{\flat} \hrar \left(\Del^n\right)^{\flat} $$
is a trivial MS-cofibration and that $f$ is an MS-fibration.
\end{enumerate}
\end{proof}

\begin{cor}\label{lifting-criterion}
Let $f: W \lrar Z$ be a fully-faithful marked fibration. Let $X$ be a marked semi-simplicial space and $g: X \lrar Z$ a map. Then every lift $\wtl{g_0}:X_0 \lrar W_0$ of $g_0$ extends to a lift $\wtl{g}: X \lrar W$ of $g$.
\end{cor}

Proposition~\ref{lift-character} allows us in particular to obtain the following description of fully-faithful maps between marked semiSegal spaces, relating them to the classical meaning of the notion:
\begin{cor}
Let $f: W \lrar Z$ be a map of marked semiSegal spaces. Then $f$ is fully-faithful if and only if $f$ induces weak equivalences
$$ \Map_{W}(x,y) \x{\simeq}{\lrar} \Map_{Z}(f_0(x),f_0(y)) $$
and
$$ \Map^+_{W}(x,y) \x{\simeq}{\lrar} \Map^+_{Z}(f_0(x),f_0(y)) .$$
\end{cor}
\begin{proof}
Factor $f$ as $W \x{f'}{\lrar} W' \x{f''}{\lrar} Z$ where $f'$ is a trivial marked cofibration and $f''$ is a marked fibration. Then $W'$ is also a marked semiSegal space and $f$ is fully-faithful if and only if $f''$ is fully-faithful. Applying Proposition~\ref{lift-character} to $f''$ one obtains the desired result.
\end{proof}

The following corollary of Proposition~\ref{lift-character} will be useful later:
\begin{cor}\label{c:lift-character-2}
Let $f: X \lrar Y$ be a map of marked semi-simplicial spaces. Then the following assertions are equivalent:
\begin{enumerate}
\item
For every MS-fibration $p:W \lrar Z$ and every MS-cofibration $g: X' \hrar Y'$ such that $g_0: X_0' \hrar Y_0'$ is a weak equivalence the induced map
$$ W^{Y'} \lrar Z^{Y'} \times_{Z^{X'}} W^{X'} $$
satisfies the right lifting property with respect to $f$.
\item
For every MS-cofibration $g: X' \hrar Y'$ such that $g_0: X_0' \hrar Y_0'$ is a weak equivalence the induced map
$$ \left[X \otimes Y'\right] \coprod_{X \otimes X'} \left[Y \otimes X'\right] \lrar Y \otimes Y' $$
is a trivial MS-cofibration.
\item
For every MS-fibration $p:W \lrar Z$ the induced map
$$ W^Y \lrar Z^Y \times_{Z^X} W^X $$
is a fully-faithful MS-fibration.
\item
$f$ is an MS-cofibration and the maps
\begin{equation}\label{e:char4a}
\left[X \otimes \left(\Del^1\right)^{\flat}\right] \coprod_{X \otimes \partial\Del^1} \left[Y \otimes \partial\Del^1\right] \lrar \left[Y \otimes \left(\Del^1\right)^{\flat}\right]
\end{equation}
\begin{equation}\label{e:char4b}
\left[X \otimes \left(\Del^1\right)^{\sharp}\right] \coprod_{X \otimes \partial\Del^1} \left[Y \otimes \partial\Del^1\right] \lrar \left[Y \otimes \left(\Del^1\right)^{\sharp}\right]
\end{equation}
are trivial MS-cofibrations.
\end{enumerate}
\end{cor}
\begin{proof}
The equivalence $(1) \Leftrightarrow (2)$ follows directly from the exponential law. The equivalence $(2) \Leftrightarrow (3)$ follows from the exponential law together with Proposition~\ref{lift-character}. We will now prove that $(3) \Leftrightarrow (4)$. The direction $(4) \Rightarrow (3)$ follows from the exponential law and Proposition~\ref{lift-character}. Now assume $f$ satisfies $(3)$. Then for every trivial MS-fibration $W \lrar Z$ the induced map
$$ W^Y \lrar Z^Y \times_{Z^X} W^X $$
is a trivial MS-fibration. This implies that $f$ has the left lifting property with respect to every trivial MS-fibration and so is an MS-cofibration. The second part of $(4)$ then follows from Proposition~\ref{lift-character}.
\end{proof}

\begin{rem}\label{r:satur}
The class of maps $f: X \lrar Y$ satisfying the equivalent conditions of Corollary~\ref{c:lift-character-2} is weakly saturated in view of characterization $(1)$ and contains all trivial MS-cofibrations.
\end{rem}

\begin{rem}\label{r:iso-under}
If $f: X \lrar Y$ is a map whose underlying map of unmarked semi-simplicial spaces is an isomorphism then $f$ satisfies condition $(4)$ of Corollary~\ref{c:lift-character-2}. To see this, note that such an $f$ is automatically an MS-cofibration. In addition, the map~\ref{e:char4a} is an isomorphism and the map~\ref{e:char4b} is a triangle remarking (see Definition~\ref{triangle-remarking}).
\end{rem}

We now turn to the notion of \textbf{Dwyer-Kan equivalences}. This notion will be obtained from the notion of fully-faithful maps by requiring an appropriate analogue of "essential surjectivity". This notion is most well behaved for quasi-unital marked semiSegal spaces (see Lemma~\ref{equiv-relation}), but it will be convenient to have it defined in more generality.

\begin{define}
Let $X$ be a marked semi-simplicial space. Let $x \simeq y$ denote the weakest equivalence relation satisfying the following properties:
\begin{enumerate}
\item
If $x,y$ are in the same connected component of $X_0$ then $x \simeq y$.
\item
If there exists a marked edge $f \in X_1$ such that $d_0(f) = x$ and $d_1(f) = y$ then $x \simeq y$.
\end{enumerate}
\end{define}

\begin{lem}\label{equiv-relation}
Let $X$ be a quasi-unital marked semiSegal space and $x,y \in X_0$ two points. Then $x \simeq y$ if and only if there exists a marked morphism $x \lrar y$.
\end{lem}
\begin{proof}
This follows from the fact that the relation of having a marked morphism from $x$ to $y$ is already an equivalence relation when $X$ is a quasi-unital marked semiSegal space, and that this relation contains the relation of being in the same connected component.
\end{proof}

\begin{define}
Let $f: X \lrar Y$ be a map between semi-simplicial spaces. We will say that $f$ is a \textbf{Dwyer-Kan equivalence} (DK for short) if it is fully faithful and induces a surjective map on the set of equivalence-classes of $\simeq$.
\end{define}

We will now show that under a mild additional hypothesis a marked-fibrant semi-simplcial space which is close enough to a quasi-unital marked semiSegal space is itself a quasi-unital marked semiSegal space. This will be useful for us when constructing completions (see \S\S~\ref{ss-completion}).
\begin{lem}\label{qu-criterion}
Let $X$ be a quasi-unital marked semiSegal space, $W$ a marked-fibrant semi-simplcial space and $f: X \lrar W$ a fully-faithful map such that $f_0:X_0 \lrar W_0$ is surjective on connected components. Then $W$ is a quasi-unital marked semiSegal space and $f$ is a DK-equivalence.
\end{lem}
\begin{proof}
The fact that $f$ is a DK-equivalence follows directly from the definition. Hence it will suffice to prove that $W$ is a quasi-unital marked semiSegal space.

We can factor $f$ as $X \x{f'}{\lrar} X' \x{f''}{\lrar} W$ such that $f'$ is a trivial marked cofibration and $f''$ is a marked fibration. Then $X'$ is marked-fibrant and marked-equivalent to $X$, so that $X'$ is necessarily a quasi-unital marked semi-simplicial space and $f''$ is fully-faithful. Hence we can assume without loss of generality that $f$ itself is a marked fibration and that $f_0: X_0 \lrar W_0$ is surjective.

We start by showing that $W$ is a marked semiSegal space. Let $f: Y \lrar Z$ be a map in $\BS$ (see Definition~\ref{local-generators}). We need to show that the map
$$ \Map^+\left(Z,W\right) \lrar \Map^+\left(Y,W\right) $$
is a weak equivalence. Note that in all cases the $0$'th level map
$ f_0: Y_0 \lrar Z_0 $
is an isomorphism. Condition $(4)$ of Proposition~\ref{lift-character} then tells us that
$$ \xymatrix{
\Map^+(Z,X) \ar[r]\ar[d] &  \Map^+\left(Z,W\right)\ar[d] \\
\Map^+(Y,X) \ar[r] & \Map^+\left(Y,W\right) \\
}$$
is homotopy Cartesian. Furthermore, all maps appearing in this diagrams are fibrations. We wish to show that the right vertical fibration is trivial. Since $X$ is a marked semiSegal space the left vertical Kan fibration is trivial. Now since Kan fibrations are trivial if and only if all their fibers are contractible it will be enough to show that the lower horizontal map is surjective. In light of Corollary~\ref{lifting-criterion} it will be enough to show that the map
$$ \Map_{\Top}(Y_0,X_0) \lrar \Map_{\Top}\left(Y_0,W_0\right) $$
is surjective. But this just follows from the fact that $Y_0$ is discrete and the map
$ X_0 \lrar W_0 $
is surjective.

We now show that $W$ is quasi-unital. According to Lemma~\ref{qu-rlp} we need to show that the map $W \lrar T$ has the right lifting property with respect to the maps $C_0,C_1,C_2$. Let $D_i$ be the domain of $C_i$. Since the $D_i$'s are levelwise discrete and the map $X_0 \lrar W_0$ is surjective we can use Corollary~\ref{lifting-criterion} in order to lift any map $D_i \lrar W$ to a map $D_i \lrar X$. Since $X$ is quasi-unital it satisfies the right lifting property with respect to each $C_i$ and so the result follows.
\end{proof}

\subsection{ Q-fibrations and Q-anodyne maps }\label{ss-Q-anodyne}

In the beginning of \S~\ref{ss-marked-semiSegal} we saw that the property of being quasi-unital can be expressed as a certain right lifting property (see Lemma~\ref{qu-rlp}). This idea leads one to define the following relative version of quasi-unitality:

\begin{define}
Let $f: W \lrar Z$ be an MS-fibration of marked semi-simplicial spaces. We will say that $f$ is a \textbf{Q-fibration} if it satisfies the right lifting property with respect to the maps $C_0,C_1,C_2$ above.
\end{define}

\begin{example}
Let $W$ be a quasi-unital marked semiSegal space. Then the natural map $W \lrar T$ is a Q-fibration.
\end{example}

The notion of Q-fibrations has a left-hand-side counterpart:
\begin{define}
Let $f: X \lrar Y$ be an MS-cofibration. We will say that $f$ is \textbf{Q-anodyne} if it satisfies the left lifting property with respect to all Q-fibrations.
\end{define}

\begin{example}
Any trivial MS-cofibration is Q-anodyne.
\end{example}

\begin{rem}\label{r:saturated}
The maps $C_0,C_1$ and $C_2$ are Q-anodyne. In fact, since $\Top^{\Del^{\op}_s}_+$ is presentable one can identify the collection of all Q-anodyne maps with the weakly saturated class of maps generated from $C_0,C_1,C_2$.
\end{rem}

\begin{rem}\label{cofully-pushouts}
Since the class of Q-anodyne maps is weakly saturated and contains all trivial MS-cofibrations it is also closed under certain \textbf{homotopy pushouts}. More precisely, if we have a homotopy pushout square
$$ \xymatrix{
X \ar[r]\ar[d] & Z \ar[d] \\
Y \ar[r] & W \\
}$$
in $\Seg_s$ such that the left vertical map is Q-anodyne then the lower right vertical map will be Q-anodyne as well as long as the square is \textbf{Reedy cofibrant}, i.e., as long as the induced map
$ Y \coprod_{X} Z \lrar W $
is an MS-cofibration.
\end{rem}

\begin{rem}
Although one cannot identify $\QsS$ with the subcategory of fibrant objects in some localization of $\Top^{\Del^{\op}_s}_+$, one can still associate with $\QsS$ the \textbf{weak factorization system} formed by Q-fibrations and Q-anodyne maps. Although not part of a model category, it enables many model categorical arguments and manipulations. The purpose of this section is to exploit this point of view to obtain results which will be used in the next section.
\end{rem}

\begin{lem}\label{l:q-is-lift}
Let $f: X \lrar Y$ be a Q-anodyne map. Then $f$ satisfies the equivalent conditions of Corollary~\ref{c:lift-character-2}.
\end{lem}
\begin{proof}
In view of Remark~\ref{r:satur} it will suffice to prove that the $C_i$' satisfy condition $(4)$ of Corollary~\ref{c:lift-character-2}. For $C_2$ this is a special case of Remark~\ref{r:iso-under}. For $C_i,i=0,1$ we need to check that the maps
$$ \left[\left(\Del^1\right)^{\sharp} \times \partial\Del^1\right] \coprod_{\Del^{\{i\}} \times \partial\Del^1} \left[\Del^{\{i\}} \times \left(\Del^1\right)^{\flat}\right] \lrar \left(\Del^1\right)^{\sharp} \otimes \left(\Del^1\right)^{\flat} $$
and
$$ \left[\left(\Del^1\right)^{\sharp} \times \partial \Del^1\right] \coprod_{\Del^{\{i\}} \times \partial \Del^1} \left[\Del^{\{i\}} \times \left(\Del^1\right)^{\sharp}\right] \lrar \left(\Del^1\right)^{\sharp} \otimes \left(\Del^1\right)^{\sharp} $$
are trivial MS-cofibrations. Now in the first map the right hand side can be obtained from the left hand side by performing two pushouts along admissible horn inclusions of dimension $2$. In the second map one needs to perform in addition a triangle remarking (see Definition~\ref{triangle-remarking}).

\end{proof}

\begin{cor}\label{Q-anodyne-lem}
Let $f: X \lrar Y$ be a Q-anodyne map and let $p:W \lrar Z$ be a Q-fibration. Then the induced map
$$ f^p: W^Y \lrar Z^Y \times_{Z^X} W^X $$
is a DK-equivalence.
\end{cor}
\begin{proof}
From Lemma~\ref{l:q-is-lift} it follows that $f^p$ is a fully-faithful MS-fibration. Since $Q$-anodyne maps satisfy the left lifting property with respect to $Q$-fibrations we get that the fibration
$$ f^p_0: \left(W^Y\right)_0 \lrar \left(Z^Y \times_{Z^X} W^X\right)_0 = \left(Z^Y\right)_0 \times_{\left(Z^X\right)_0} \left(W^X\right)_0 $$
is surjective and so the result follows.
\end{proof}

We will now apply some of the ideas collected so far in order to prove that the full subcategory $\QsS \subseteq \Seg_s^{\fib}$ is closed under mapping objects. In fact, we will prove that for any semi-simplicial space $A$ and any quasi-unital marked semiSegal space $W$ the mapping object $W^A$ is quasi-unital.

\begin{prop}\label{mapping-is-qu}
Let $W$ be a quasi-unital marked semiSegal space and $A$ a marked semi-simplicial space. Then $W^A$ is quasi-unital.
\end{prop}
\begin{proof}
Let $f: X \lrar Y$ be a Q-anodyne map. To show that the terminal map $W^A \lrar T$ satisfies the right lifting property with respect to $f$ is equivalent to showing that the map $f \otimes A: X \otimes A \lrar Y \otimes A$ satisfy the left lifting property with respect to $W \lrar T$.

Consider the space $A_0$ as a marked semi-simplicial space concentrated in degree $0$. We have a natural inclusion $g:A_0 \hrar A$ such that $g_0$ is an isomorphism. By Lemma~\ref{l:q-is-lift} we get that the natural map
$$ \left[X \otimes A\right] \coprod_{X \otimes A_0} \left[Y \otimes A_0\right] \lrar Y \otimes A $$
is a trivial MS-cofibration. Hence it will suffice to prove that the map $f \otimes A_0$ satisfies the left lifting property with respect to $W \lrar T$, and it will suffice to do so for $f=C_i$,

Let us start with $C_2$. Let $A_{0,0}$ be the set of vertices of $A_0$. Then $C_2 \otimes A_{0,0}$ satisfies the left lifting property with respect to $W \lrar T$. The same claim for $C_2 \otimes A_0$ follows from the fact that
$$ \left[A_0 \otimes \left(\Del^3,A\right)\right] \coprod_{A_{0,0} \otimes \left(\Del^3,A\right)} \left[A_{0,0} \otimes \left(\Del^3\right)^{\sharp}\right] \hrar A_0 \otimes \left(\Del^3\right)^{\sharp} $$
is a marked equivalence (where $A = \{\Del^{\{0,2\}},\Del^{\{1,3\}}\}$).

Let us now prove the case $C_i$ for $i=0,1$. Unwinding the definitions we need to show that the map of spaces $d_i:W^{\inv}_1 \lrar W_0$ satisfies the right lifting property with respect to any map of the form $\emptyset \hrar A_0$, or, equivalently, admits a section.

Let
$$ W^{\aut}_1  = \left\{f \in  W^{\inv}_1 | d_0(f) = d_1(f)\right\} \subseteq W^{\inv}_1 $$
be the subspace of self equivalences. It will be enough to show that the map $d:W^{\aut}_1 \lrar W_0$ (induced by either $d_0$ or $d_1$) admits a section. Since this claim involves only marked edges it will convenient to switch to the maximal semiKan space $Z = \wtl{W} \subseteq W$. In particular, we want to show that the natural map $Z^{\aut}_1 \lrar Z$ admits a section.

Consider the Kan replacement $\what{|Z|}$ of the realization of $Z$. Let $\what{|Z|}^{S^1}$ be the space of continuous paths
$ \gam: S^1 \lrar \what{|Z|} $
and let $p: \what{|Z|}^{S^1} \lrar \what{|Z|}$ be the map $p(\gam) = \gam(1)$. Consider the commutative diagram
$$ \xymatrix{
Z^{\aut}_1 \ar[r]\ar^{d}[d] & \what{|Z|}^{S^1} \ar^{p}[d] \\
Z_0  \ar[r] & \what{|Z|}\\
}$$
The vertical maps in this square are fibrations and by Theorem~\ref{realization-of-groupoid} the square is homotopy Cartesian. Now since the right vertical map admits a section (given by choosing for each $x$ the constant path at $x$) we get that the left vertical map admits a section as well. This finishes the proof of Proposition~\ref{mapping-is-qu}.
\end{proof}

\begin{cor}\label{c:section}
Let $f: X \lrar Y$ be a Q-anodyne map and $W$ a quasi-unital marked semiSegal space. Then the map $W^Y \lrar W^X$ admits a section.
\end{cor}
\begin{proof}
It is enough to prove that the map $W^Y \lrar W^X$ satisfies the right lifting property with respect to every map of the form $\emptyset \hrar A$. But this follows from Proposition~\ref{mapping-is-qu} as $W^A$ is quasi-unital.
\end{proof}

Our final goal of this subsection is to show that the following types of maps are Q-anodyne. Let $f: [n] \lrar [k]$ be a \textbf{surjective} map in $\Del$ and let $h: [k] \lrar [n]$ be a section of $f$. Let $M \subseteq (\Sp^k)_1$ be a marking on the $k$-spine and let $\wtl{M} \subseteq (\Del^k)_1$ be the marking generated from it, i.e., the smallest set containing $M$ which is closed under $2$-out-of-$3$. Let $M_f \subseteq \left(\Sp^n\right)_1$ be the set of all pairs $\{i,i+1\}$ such that either $f(i) = f(i+1)$ or $\Del^{\{f(i),f(i+1)\}}$ is in $M$ and let $\wtl{M}_f \subseteq (\Del^n)_1$ be the marking generated from it. We wish to prove the following:

\begin{prop}\label{Q-anodyne}
In the notation above, the map
$$ h: \left(\Del^k,\wtl{M}\right) \lrar \left(\Del^n,\wtl{M}_f\right) $$
is Q-anodyne.
\end{prop}

\begin{proof}

For each $i = 0,...,n$ let $S_i \subseteq \Del^n$ be the $1$-dimensional sub semi-simplicial set containing all the vertices and all the edges of the form $\Del^{\{j,j+1\}}$ such that
$ f(j) = f(j+1) = i $.
Then clearly the inclusion
$ \Del^{\{h(i)\}} \subseteq S_i $
is Q-anodyne. Let $S \subseteq \Del^n$ be the (disjoint) union of all the $S_i$'s.

Let $h_1,h_2$ be two sections of $f$. We will define the sub marked semi-simplicial set $T(h_1,h_2) \subseteq \left(\Del^n,\wtl{M}_f\right)$ to be the (not necessarily disjoint) union of $S^{\sharp} \subseteq \left(\Del^n,\wtl{M}_f\right)$ and all edges of the form $\Del^{\{h_1(i),h_2(i+1)\}}$ for $i=0,...,n$. In particular,
$$ T(h_1,h_2) = \left(S \bigcup_i \Del^{\{h_1(i),h_2(i+1)\}},N\right) \subseteq \left(\Del^n,\wtl{M}_f\right) $$
where $N$ is the marking induced from $\wtl{M}_f$, i.e., $N$ contains all the edges of $S$ and the edges $\Del^{\{h_1(i),h_2(i+1)\}}$ for which $\Del^{\{i,i+1\}}$ is in $M$. Note that when $h_1=h_2=h$ we have
$$ T(h,h) = S^{\sharp} \coprod_{\Del^0 \times \{0,...,k\}} \left(h(\Sp^k),h(M)\right) .$$
Now consider the commutative square
$$ \xymatrix{
\left(\Sp^k,M\right) \ar^{h}[r]\ar[d] & \left(\Del^k,\wtl{M}\right)  \ar[d] \\
T(h,h)\ar[r] & \left(\Del^n,\wtl{M}_f\right) \\
}$$
Since $h:\Del^0 \times \{0,...,k\} \hrar S^{\sharp}$ is Q-anodyne we get that the top horizontal row is Q-anodyne. In light of Remark~\ref{cofully-pushouts} it will now be enough to show that this square is a Reedy cofibrant homotopy pushout square in the marked Segal model structure. As Reedy cofibrancy is immediate it will be enoush to show that both horizontal maps are trivial MS-cofibrations. The top horizontal map is very easy:

\begin{lem}\label{spine-example}
The inclusion
$ \iota:\left(\Sp^k,M\right) \hrar \left(\Del^k,\wtl{M}\right) $
is a trivial MS-cofibration.
\end{lem}
\begin{proof}
Factor $\iota$ as
$$ \left(\Sp^k,M\right) \x{\iota'}{\hrar} \left(\Del^k,M\right) \x{\iota''}{\hrar} \left(\Del^k,\wtl{M}\right) .$$
Then $\iota'$ is a pushout along the trivial MS-cofibration $\left(\Sp^k\right)^{\flat} \hrar \left(\Del^k\right)^{\flat}$ and $\iota''$ is a triangle remarking (Definition~\ref{triangle-remarking}).
\end{proof}

To show that bottom horizontal map is a trivial MS-cofibration it will be convenient to prove a slightly stronger lemma:
\begin{lem}\label{T-h1-h2}
For every two sections $h_1,h_2$ of $f$ the inclusion
$$ T(h_1,h_2) \subseteq \left(\Del^m,\wtl{M}_f\right) $$
is a trivial MS-cofibration.
\end{lem}
\begin{proof}
We begin by arguing that that it is enough to prove the lemma for just one pair of sections $h_1,h_2$. We say that two pairs $(h_1,h_2),(h_1',h_2')$ are \textbf{neighbours} if
$$ \sum_{i=0}^{n} |h_1(i)-h_1'(i)| + |h_2(i)-h_2'(i)| = 1 .$$
It is not hard to see that the resulting neighbouring graph is connected, i.e., that we can get from any pair $(h_1,h_2)$ to any other pair $(h_1',h_2')$ by a sequence of pairs such that each consecutive couple of pairs are neighbours. Hence it is enough to show that property of
$ T(h_1,h_2) \hrar \left(\Del^m,\wtl{M}_f\right) $
being a trivial MS-cofibration respects the neighbourhood relation. To see why this is true observe that if $(h_1,h_2)$ and $(h_1,h_2')$ are neighbours then one can add to $T(h_1,h_2)$ a single triangle $\sig \subseteq \Del^m$ such that
$ R \x{\df}{=} T(h_1,h_2) \cup \sig $
contains $T(h_1',h_2')$ and such that $R$ can be obtained from either $T(h_1,h_2),T(h_1',h_2')$ by performing a pushout along a $2$-dimensional admissible marked horn inclusion and possibly a remarking. Hence the claim for either $T(h_1,h_2)$ or $T(h_1',h_2')$ is equivalent to
$ R \hrar \left(\Del^m,\wtl{M}_f\right) $
being a trivial MS-cofibration.

Now that we know that it is enough to prove for a single choice of $(h_1,h_2)$ let us choose the pair
$ h_{\max}(i) = \max(f^{-1}(i))$ and $h_{\min}(i) = \min(f^{-1}(i))$. Then we see that
$ T(h_{\max},h_{\min}) = \left(\Sp^m,M_f\right)  $
and the map
$ \left(\Sp^m,M_f\right) \hrar \left(\Del^m,\wtl{M}_f\right) $
is a trivial MS-cofibration from Lemma~\ref{spine-example}.
\end{proof}

This finishes the proof of Proposition~\ref{Q-anodyne}.
\end{proof}

\section{ Complete marked semiSegal spaces }\label{s-complete-semiSegal}

In this section we will further localize the model category $\Seg_s$ to obtain our target model category $\Comp_s$. We will then show that the full subcategory of fibrant objects in $\Comp_s$ is a model for the localization of $\QsS$ by DK-equivalence, and hence $\Comp_s$ is a model category for the homotopy theory of quasi-unital $\infty$-categories. Finally, we will prove the main theorem of this paper by showing that $\Comp_s$ is Quillen euqivalent to the Rezk's model category $\Comp$, and that this Quillen equivalence preserves mapping objects.

We begin with a description of the fibrant objects in $\Comp_s$, which are called \textbf{complete marked semiSegal spaces}. The notion of completeness, the construction of the completion functor and many of the related proofs are inspired by their respective analogues in~\cite{rez}. We begin with the basic definition:

\begin{define}\label{d:complete}
Let $(X, M)$ be a marked semiSegal space. We will say that $X$ is \textbf{complete} if the following two conditions are satisfied:
\begin{enumerate}
\item
The inclusion $M \subseteq X^{\inv}_1$ is a weak equivalence (hence an isomorphism of in view of Lemma~\ref{marked-fibrant}).
\item
The restricted maps $d_0:M \lrar X_0$ and $d_1:M \lrar X_0$ are both homotopy equivalences.
\end{enumerate}

\end{define}

\begin{rem}
If $W$ is a complete marked semiSegal space then $\wtl{W}$ is complete as well.
\end{rem}

\begin{rem}\label{r:homotopy-constant}
If $W$ is a marked semiKan space then $W$ is complete if and only if it is homotopy-constant as a semi-simplicial space. This follows from the Segal condition and the fact that $\Del_s$ is weakly contractible.
\end{rem}

An important observation is that any complete semiSegal space is quasi-unital: by condition $(1)$ of Definition~\ref{d:complete} all the equivalences are marked and from condition $(2)$ we get that every object $x \in X_0$ admits a marked morphism of the form $f: x \lrar y$ for some $y$.

Let $\CsS \subseteq \QsS$ denote the full topological subcategory spanned by complete marked semiSegal spaces. We will show in the following sections that the topological category $\CsS$ serves as a model for the (left) localization of $\QsS$ by DK-equivalences. In particular, $\CsS$ is a model for the \textbf{homotopy theory of quasi-unital $\infty$-categories}.

\subsection{ The complete semiSegal model structure }
Our purpose in this section is to show that the topological category $\CsS$ of complete marked semiSegal spaces can be identified with the full subcategory of fibrant objects in a suitable localization of $\Seg_s$.

Recall the maps $C_0,C_1,C_2$ introduced in~\ref{ss-Q-anodyne}. Then we have the following simple observation.
\begin{lem}
Let $(X,M)$ be a marked semiSegal space. Then $X$ is local with respect to $C_0,C_1,C_2$ if and only if $X$ is complete.
\end{lem}
\begin{proof}
First of all it is clear that $X$ is local with respect to $C_0,C_1$ if and only if $X$ satisfies condition $(2)$ of definition~\ref{d:complete}. Now assume that $X$ also satisfies condition $(1)$ of that definition. Then by Lemma~\ref{inv-closed-composition} we get that $M = X^{\inv}_1$ satisfies the $2$-out-of-$6$ property.

Now assume that $X$ is local with respect to $C_0,C_1,C_2$ so that $d_0,d_1: M \lrar X_0$ are equivalences (in particular, every object admits a marked edge into it and a marked edge out of it). It will suffice to show that every equivalence in $X$ is marked. Let $f: x \lrar y$ be an equivalence in $X$ and let $g: x \lrar z$ and $h: w \lrar y$ be arbitrary marked edges. Since $f$ is invertible we can embed these edges in a diagram of the form
$$
\xymatrix{
& x \ar^{g}[rr]\ar^{f}[dr] && z  \\
w \ar_{h}[rr]\ar@{-->}[ur] && y \ar@{-->}[ur] &
}$$
This diagram can in turn be extended to a map $\vphi:\left(\Del^3,A\right) \lrar X$ which sends $\Del^{\{1,2\}}$ to $f$. Since $X$ is local with respect to $C_2$ this implies that $f$ must be marked.
\end{proof}

Since $\Seg_s$ is combinatorial and left proper the left Bousfield localization of $\Seg_s$ with respect to the maps $C_0,C_1,C_2$ exists. We will denote the resulting localization by $\Comp_s$. Then $\Comp_s$ is a combinatorial model category satisfying the following properties:
\begin{enumerate}
\item
A map $f: X \lrar Y$ of marked semi-simplicial spaces is an equivalence in $\Comp_s$ if and only if for every complete marked semiSegal space $W$ the induced map
$$ \Map_+(Y,W) \lrar \Map_+(X,W) $$
is a weak equivalence.
\item
A map $f: X \lrar Y$ of marked semi-simplicial spaces is a cofibration in $\Comp_s$ if and only if it is a cofibration in $\Seg_s$ (i.e. a levelwise inclusion).
\item
An marked semi-simplicial space $W$ is fibrant in $\Comp_s$ if and only if it is a complete marked semiSegal space.
\end{enumerate}

\begin{thm}
The complete model structure is compatible with the marked monoidal product $\otimes$. In particular, the localization Quillen adjunction
$$ \xymatrix{
\Seg_+^{\Del_s^{\op}} \ar@<0.5ex>[r]^{\Id} & \Comp_s \ar@<0.5ex>[l]^{\Id} \\
}$$
is strongly monoidal and $\Comp_s$ inherits the simplicial structure of $\Seg_s$.
\end{thm}
\begin{proof}
Arguing as in~\cite{rez} Proposition 9.2, we see that it will be enough to establish the following:
\begin{prop}\label{mapping-into-complete}
Let $X$ be a marked semi-simplicial space and $W$ a complete marked semiSegal space. Then $W^X$ is complete.
\end{prop}
\begin{proof}
From Proposition~\ref{mapping-is-qu} we get that $W^X$ is quasi-unital. In particular, the marked edges of $W^X$ are exactly the equivalences. Hence it will suffice to prove that $W^X$ satisfies condition $(2)$ of Definition~\ref{d:complete}.

For $i=0,1$ consider the restriction map
$$ p^i: W^{\left(\Del^1\right)^{\sharp}} \lrar W^{\left(\Del^0\right)} = W .$$
Since $W$ is complete we get by definition that the maps
$$ p^i_0: \left(W^{\left(\Del^1\right)^{\sharp}}\right)_0 \lrar W_0 $$
are weak equivalences. By Corollary~\ref{Q-anodyne-lem} we get that $p^i$ is also a DK-equivalence and hence a marked equivalence. Using the exponential law this implies that the restriction map
$$ \Map^+\left(X \otimes \left(\Del^1\right)^{\sharp},W\right) \lrar \Map^+\left(X
\otimes \Del^{\{i\}},W\right) $$
is a weak equivalence. Applying the exponential law again we get that $W^X$ satisfies condition $(2)$ of Definition~\ref{d:complete}.
\end{proof}
\end{proof}

\begin{rem}
In the above notation, since $W^X$ is complete we get that $\wtl{W^X}$ is complete as well. In particular, $\wtl{W^X}$ is a homotopy-constant marked semi-simplicial space (see Remark~\ref{r:homotopy-constant}).
\end{rem}

\subsection{ Completion }\label{ss-completion}

In this section we will prove that $\CsS$ is a model for the localization of $\QsS$ by DK-equivalences. Formally speaking (see Definition $5.2.7.2$ and Proposition $5.2.7.12$ of~\cite{higher-topos}) what we need to show is that there exists a functor
$$ \what{\bullet}: \QsS \lrar \CsS $$
such that:
\begin{enumerate}
\item
$\what{\bullet}$ is homotopy left adjoint to the inclusion $\CsS \subseteq \QsS$.
\item
A map in $\QsS$ is a DK-equivalence if and only if its image under $\what{\bullet}$ is a homotopy equivalence.
\end{enumerate}
The functor $\what{\bullet}$ will be called the \textbf{completion} functor, and can be constructed as follows. Let $X$ be a quasi-unital marked semiSegal space. Consider the bi-semi-simplicial spaces $X_{\bullet,\bullet},Y_{\bullet,\bullet}: \Del^{\op}_s \times \Del^{\op}_s \lrar \Top $ given by
$$  X_{n,m} = \Map^{+}\left(\left(\Del^n\right)^{\flat} \otimes \left(\Del^m\right)^{\sharp}, X\right) $$
and
$$ Y_{n,m} = \Map((\Del^n)^{\sharp} \otimes (\Del^m)^{\sharp},X) .$$
We define the marked semi-simplicial space $\left(\ovl{X},M\right)$ by setting
$$ \ovl{X}_n = |X_{n,\bullet}| $$
and
$$ M = |Y_{1,\bullet}| \subseteq |X_{1,\bullet}| .$$
We then define the \textbf{completion} $\what{X}$ of $X$ to be the \textbf{marked-fibrant replacement} of $\left(\ovl{X},M\right)$.

\begin{thm}\label{completion}
Let $X$ be a quasi-unital marked semiSegal space. Then
\begin{enumerate}
\item
$\what{X}$ is a complete marked semiSegal space.
\item
The natural map $X \lrar \what{X} $ is a DK-equivalence.
\end{enumerate}
\end{thm}
\begin{proof}
Let $n \geq 0$ be an integer. From Proposition~\ref{Q-anodyne} we conclude that for each map $f: [k] \lrar [m]$ in $\Del_s$ the map
$$ f^*:X_{\bullet,m} \lrar X_{\bullet,k} $$
is a DK-equivalence and in particular fully-faithful. Hence the induced square
$$ \xymatrix{
X_{n,m} \ar[r]\ar^{f_n^*}[d] & (X_{0,m})^{n+1} \ar^{(f_0^*)^{n+1}}[d] \\
X_{n,k} \ar[r]       & (X_{0,k})^{n+1} \\
}$$
is homotopy Cartesian. From Corollary~\ref{products} the natural map
$$ \left|X_{0,\bullet}^{n+1}\right| \lrar \left|X_{0,\bullet}\right|^{n+1} $$
is a weak equivalence and so Puppe's Theorem (see Theorem~\ref{pupe}) implies that the square
$$ \xymatrix{
X_{n,0} \ar[r]\ar[d] & (X_{0,0})^{n+1}  \ar[d] \\
|X_{n,\bullet}| \ar[r] & \left|X_{0,\bullet}\right|^{n+1} \\
}$$
is homotopy Cartesian. The same argument with $Y_{\bullet,\bullet}$ instead of $X_{\bullet,\bullet}$ shows that the square
$$ \xymatrix{
Y_{1,0} \ar[r]\ar[d] & (Y_{0,0})^{2}  \ar[d] \\
|Y_{1,\bullet}| \ar[r] & \left|X_{0,\bullet}\right|^{2} \\
}$$
is homotopy Cartesian. This implies that the map $X \lrar \ovl{X}$ (and hence also the map $X \lrar \what{X}$) is fully-faithful. We now observe that the map $X_0 \lrar \what{X}_0$ is surjective on connected components. Since $\what{X}$ is marked-fibrant we deduce from Lemma~\ref{qu-criterion} that $\what{X}$ is a quasi-unital marked semiSegal space and the map $X \lrar \what{X}$ is a DK-equivalence.

It is left to show that $\what{X}$ is complete. Since $\what{X}$ is quasi-unital we know that all invertible edges in $\what{X}$ are marked. Hence it will suffice to show that the maps
$$ |d_0|,|d_1|: \left|Y_{1,\bullet}\right| \lrar \left|Y_{0,\bullet}\right| = \ovl{X}_0 $$
are weak equivalences. But this follows from Corollary~\ref{equiv-preserves-realization} since the maps $d_0,d_1: Y_{1,\bullet} \lrar Y_{0,\bullet}$ are DK-equivalences of marked semiKan spaces.

\end{proof}

Our goal now is to show that the completion functor is a localization functor with respect to DK-equivalences. We start by showing that the notions of DK-equivalence and marked equivalence coincide in $\CsS$. Since complete marked semiSegal spaces are marked-fibrant the notion of a marked equivalence is the same as the notion of an equivalence in $\CsS$ as a \textbf{topological category}:
\begin{prop}\label{complete-DK}
Let $f: X \lrar Y$ be a DK-equivalence between complete marked semiSegal spaces. Then $f$ is a marked equivalence.
\end{prop}
\begin{proof}
Since $f$ is in particular fully-faithful it will be enough to show that the map
$ f_0: X_0 \lrar Y_0 $
is a weak equivalence. Let
$ \wtl{f}: \wtl{X} \lrar \wtl{Y} $
be the induced map between the corresponding maximal sub semiKan spaces. Then clearly $\wtl{f}$ is a DK-equivalence as well. From Corollary~\ref{equiv-preserves-realization} it follows that the induced map
$$ \left|\wtl{f}\right|: \left|\wtl{X}\right| \lrar \left|\wtl{Y}\right| $$
is a weak equivalence. But since $X,Y$ are complete their corresponding maximal semiKan space are homotopy-constant and so their realization is naturally equivalent to their space of objects. It follows that $f_0$ is an equivalence and we are done.
\end{proof}

We are now ready to prove the main theorem of this subsection:

\begin{thm}\label{t:localization}
The completion functor $\what{\bullet}: \QsS \lrar \CsS$ exhibits $\CsS$ as the left localization of $\QsS$ with respect to DK-equivalences.
\end{thm}

\begin{proof}
From the second part of Theorem~\ref{completion} we get that a map $f: X \lrar Y$ of quasi-unital semiSegal spaces is a DK-equivalence if and only if
$ \what{f}: \what{X} \lrar \what{Y} $
is a DK-equivalence. In view of Proposition~\ref{complete-DK} we deduce that the collection of maps sent by $\what{\bullet}$ to equivalences are precisely the DK-equivalences. Hence it is left to prove that $\what{\bullet}$ is indeed a homotopy left adjoint to the inclusion $\CsS \hrar \QsS$. For this it will be enough to show that the natural map
$$ X \lrar \what{X} $$
is a weak equivalence in $\Comp_s$ (and hence induces an equivalence on mapping spaces into complete semiSegal spaces). Now the marked semi-simplicial space $\what{X}$ is the homotopy colimit of the $\Del^{\op}_s$-diagram $[m] \mapsto X_{\bullet,m}$ where $X_{\bullet,0} = X$. Since $\Del_s$ is weakly contractible it will suffice to show that for each $\rho: [k] \lrar [n]$ in $\Del_s$ the natural map
$$ \rho^*:X_{\bullet,n} = X^{\left(\Del^n\right)^{\sharp}} \lrar X^{\left(\Del^k\right)^{\sharp}} = X_{\bullet,k} $$
is a weak equivalence in $\Comp_s$. Now from Proposition~\ref{Q-anodyne} we know that the inclusion
$$ \left(\Del^k\right)^{\sharp} \hrar \left(\Del^n\right)^{\sharp} $$
is Q-anodyne. Hence it will be enough to prove the following key assertion:

\begin{thm}\label{Q-anodyne-lem-2}
Let $f: X \hrar Y$ be a Q-anodyne map and let $W$ be a quasi-unital marked semiSegal space. Then the map
$ f^*:W^Y \lrar W^X $
is an equivalence in $\Comp_s$.
\end{thm}

The remainder of this section is devoted to proving Theorem~\ref{Q-anodyne-lem-2}. For this we will need to have some way to spot weak equivalences in $\Comp_s$. This will be achieved using a weak notion of cylinder object:

\begin{define}\label{d:cylinder}
Let $\M$ be a model category and $X \in \M$ and object. We will say that a cofibration of the form
$$ X \coprod X \x{d_0 \coprod d_1}{\lrar}  IX $$
exhibits $IX$ as a \textbf{weak cylinder object} for $X$ if the two maps $d_0,d_1: X \lrar IX$ are weak equivalences which become equal in $\Ho(\M)$. Given a weak cylinder object as above and two maps $f,g: X \lrar Y$ we will say that $f,g$ are \textbf{homotopic} via $IX$ if the corresponding map
$$ X \coprod X \x{f \coprod g}{\lrar}  Y $$
factors through $IX$. This notion is in general stronger then $f,g$ being equal in $\Ho(\M)$.
\end{define}

Our reason for introducing this notion is that in the model category $\Comp_s$ we have very natural choices for weak cylinder object, namely:
\begin{lem}
Let $X$ be a marked semi-simplicial space. Then the natural map
$$ X \coprod X \hrar X \otimes \left(\Del^1\right)^{\sharp} $$
exhibits $X \otimes \left(\Del^1\right)^{\sharp}$ as a weak cylinder object of $X$.
\end{lem}
\begin{proof}
Clearly the map in question is a cofibration. It will hence be enough to show that the two maps $d_0,d_1:\Del^0 \lrar \left(\Del^1\right)^{\sharp}$ are equal in $\Ho(\Comp_s)$. Let $\iota:\left(\Del^1\right)^{\sharp} \lrar W$ be a fibrant replacement of $\left(\Del^1\right)^{\sharp}$ in $\Comp_s$. We need to show that $\iota \circ d_0$ and $\iota \circ d_1$ are in the same connected component of $\Map_+(\Del^0,W) = W_0$. But this is clear because the map $\iota$ determines a marked edge from $\iota \circ d_0$ to $\iota \circ d_1$ and $W$ is complete (so that $\wtl{W}$ is homotopy-constant).
\end{proof}

The notion of homotopy between maps which is associated to the above choice of weak cylinder objects will be called \textbf{$\left(\Del^1\right)^{\sharp}$-homotopy}. There is a corresponding notion of a \textbf{$\left(\Del^1\right)^{\sharp}$-homotopy equivalence}, which in general is stronger then being a weak equivalence in $\Comp_s$. These types of equivalences are analogous to the notion of \textbf{categorical equivalences} in~\cite{rez}.

We will apply our construction of cylinder objects in order to prove Theorem~\ref{Q-anodyne-lem-2}. From Proposition~\ref{Q-anodyne-lem} and Corollary~\ref{c:section} we know that if $f: X \lrar Y$ is Q-anodyne then the map $W^Y \lrar W^X$ is a DK-equivalence which admits a section. Hence Theorem~\ref{Q-anodyne-lem-2} will follow from the following proposition:

\begin{prop}\label{DK-to-categorical}
Let $p:W \lrar Z$ be a DK-equivalence between quasi-unital marked semiSegal spaces which admits a section $g: Z \lrar W$. Then $f$ is a weak equivalence in $\Comp_s$.
\end{prop}
\begin{proof}
We can assume without loss of generality that $f$ is a marked fibration. We claim that $g$ is a homotopy inverse of $f$. On one direction the composition $f \circ g$ is the identity. We need to show that $g \circ f$ is equivalent to the identity $\Comp_s$. For this it will suffice to produce a $\left(\Del^1\right)^{\sharp}$-homotopy from $g \circ f$ to the identity, or in other words a marked edge from $g \circ f$ to the identity in $W^W$.

Since the mapping object $Z^W$ is quasi-unital (Proposition~\ref{Q-anodyne-lem}) there exists a marked edge $h \in (Z^W)_1$ from $f$ to itself. The edge $h$ corresponds to a map
$$ h: W \otimes \left(\Del^1\right)^{\sharp} \lrar Z $$
whose restriction to each $Z \otimes \Del^{\{i\}}$ is $f$. Now consider the commutative square
$$ \xymatrix{
W \times \partial \Del^1 \ar[r]\ar[d] & W \ar^{f}[d] \\
W \times \left(\Del^1\right)^{\sharp} \ar@{-->}^{\wtl{h}}[ur] \ar^{h}[r] & Z \\
}$$
where the top horizontal map is given by $(g \circ f) \coprod \Id$. Since the right vertical map is a fully-faithful marked fibration and the left vertical map is a cofibration which induces an isomorphism on the $0$'th level we get from Proposition~\ref{lift-character} that the a lift $\wtl{h}: W \times \left(\Del^1\right)^{\sharp} \lrar W$ indeed exists. Then $\wtl{h}$ gives an equivalence from $g \circ f$ to the identity in $W^W$ and we are done.

\end{proof}

This finishes the proof of Theorem~\ref{Q-anodyne-lem-2}, and hence the proof of Theorem~\ref{t:localization}.
\end{proof}

\subsection{ Proof of the main theorem }\label{ss-equivalence}

Let $\Comp$ denote Rezk's model category of complete Segal spaces (so that the underlying category of $\Comp$ is the category $\Top^{\Del^{\op}}$ of simplicial spaces). Recall the Quillen adjunction
$$ \xymatrix{
\Top^{\Del^{\op}} \ar@<0.5ex>[r]^{\F^+} & \Top^{\Del^{\op}_s}_+ \ar@<0.5ex>[l]^{\RK^+} \\
}$$
described in \S\S~\ref{sss-marked-right-kan}. The purpose of this section is to prove the following theorem, which is the main result of this paper:

\begin{thm}\label{t:main-2}
The Quillen adjunction $\F^+ \dashv \RK^+$ descends to a Quillen equivalence
\begin{equation}\label{e:adjunction}
\xymatrix{
\Comp \ar@<0.5ex>[r]^{\F^+} & \Comp_s \ar@<0.5ex>[l]^{\RK^+} \\
}
\end{equation}
\end{thm}

Note that a-priori it is not even clear that this is a Quillen adjunction. Even though $\F^+$ preserves cofibrations, to show that it preserves trivial cofibrations is equivalent to showing that $\RK^+$ maps complete marked semiSegal spaces to complete Segal spaces. Fortunately, this claim as well as the desired Quillen equivalence will both follow from the following theorem:

\begin{thm}\label{counit}
Let $X$ be a complete marked semiSegal space. Then the counit map
$$ \nu_X:\F^+\left(\RK^+\left(X\right)\right) \lrar X $$
is a marked equivalence.
\end{thm}

Before we proceed to prove Theorem~\ref{counit} let us derive two short corollaries of it, which together imply that~\ref{e:adjunction} is a Quillen equivalence. We start with the following observation:
\begin{cor}\label{rk-complete}
Assume Theorem~\ref{counit} and let $X$ be a complete marked semiSegal space. Then $\RK^+(X)$ is a complete Segal space.
\end{cor}
\begin{proof}
First since $X$ is marked-fibrant we get that $\RK^+(X)$ is Reedy fibrant. Since
$ \F^+\left(\RK^+\left(X\right)\right) \simeq X $
we get that $\RK^+\left(X\right)$ satisfies the Segal condition and hence is a Segal space. Furthermore, we get that
$ \RK^+\left(X\right)^{\inv}_1 \simeq X^{\inv}_1 $
and in particular the map
$$ d_0: \RK^+\left(X\right)^{\inv}_1 \lrar \RK^+\left(X\right)_0 $$
is a weak equivalence. Since $s_0$ is a section of $d_0$ we get that $s_0$ is a weak equivalence and hence $\RK^+\left(X\right)$ is a complete Segal space.
\end{proof}

Now Corollary~\ref{rk-complete} implies that
$$ \xymatrix{
\Comp \ar@<0.5ex>[r]^{\F^+} & \Comp_s \ar@<0.5ex>[l]^{\RK^+} \\
}$$
is indeed a Quillen adjunction. We then get a derived adjunction
$$ \xymatrix{
\Ho\left(\Comp\right) \ar@<0.5ex>[r] & \Ho\left(\Comp_s\right) \ar@<0.5ex>[l] \\
}$$
between the respective homotopy categories. Since every object in $\Comp$ is cofibrant Theorem~\ref{counit} tells us that the counit of the derived adjunction is a natural isomorphism. To show that the unit map is an isomorphism it will be enough to show that $\F^+$ \textbf{detects equivalences}, i.e. that if $f: X \lrar Y$ is a map of simplicial spaces such that $\F^+(f)$ is a weak equivalence in $\Comp_s$ then $f$ is an equivalence in $\Comp$ (the unit transformation is then an equivalence by a standard argument).

\begin{cor}
Assume Theorem~\ref{counit}. Then the functor $\F^+$ detects equivalences.
\end{cor}
\begin{proof}
By definition the equivalences in $\Comp$ are detected by mapping into complete Segal spaces. Hence the claim that $\F^+$ detects equivalences will follow once we show that every complete Segal space is in the image of $\RK^+$ (up to a levelwise equivalence).

Let $Y$ be a complete Segal space. Note that $\F^+(Y)$ is then \textbf{almost} a complete marked semiSegal space in the following sense: let $\F^{\natural}(Y)$ be the marked simplicial space which has the same underlying semi-simplicial space as $\F^+(Y)$ but whose marking is given by $M = (\F^+(Y))^{\inv}_1 = Y^{\inv}_1$. Since every degenerate edge in $Y$ is an equivalence in $\F^+(Y)$ we have an inclusion
$$ \iota:\F^+(Y) \hrar \F^{\natural}(Y) $$
which is a levelwise equivalence by definition. Since $Y$ is complete we get that $Y^{\inv}_1$ is the union of all connected components which contain degenerate edges. This means that $\iota$ is a marked equivalence. Furthermore the completeness of $Y$ implies the completeness of $\F^{\natural}(Y)$, and so $\iota$ can serve as a \textbf{fibrant replacement} of $\F^+(Y)$. The upshot of this is that $\F^+(Y)$ is \textbf{marked equivalent} to its fibrant replacement (and not just weakly equivalent in $\Comp_s$).

Now let $u_Y$ be the map given by the composition
$$ Y \lrar \RK^+\left(\F^+(Y)\right) \lrar \RK^+\left(\F^{\natural}(Y)\right)) .$$
From the discussion above we see that we can consider $u_Y$ is the derived unit in the \textbf{pre-localized} Quillen adjunction
$$ \xymatrix{
\Top^{\Del^{\op}} \ar@<0.5ex>[r]^{\F^+} & \Top^{\Del^{\op}_s}_+ \ar@<0.5ex>[l]^{\RK^+} \\
}.$$
This means that the composition
$$ \F^+(Y) \x{\F^+(u_Y)}{\lrar} \F^+\left(\RK^+\left(\F^{\natural}(Y)\right)\right) \x{\nu_{\F^{\natural}(Y)}}{\lrar} \F^{\natural}(Y) $$
is a marked equivalence. Since the second map is a marked equivalence by Theorem~\ref{counit} we get that $\F^+(u_Y)$ is a marked equivalence. This means that $u_Y$ is a levelwise equivalence and hence $Y$ is in the essential image of $\RK^+$.
\end{proof}

\begin{rem}\label{r:topological-adj}
Let $\CS \subseteq \Comp^{\fib}$ be the full topological subcategory spanned by complete Segal spaces. In light of Corollary~\ref{counit} the functors $\RK^+$ and $\F^{\natural}$ above restrict to an adjunction
$$ \xymatrix{
\CS \ar@<0.5ex>[r]^{\F^{\natural}} & \CsS \ar@<0.5ex>[l]^{\RK^+} \\
} $$
which is an \textbf{equivalence of topological categories}.
\end{rem}

We are now ready to prove Theorem~\ref{counit}:

\begin{proof}[Proof of Theorem~\ref{counit}]

We will start with a lemma which will help us compute $\RK^+\left(X\right)$ more easily by replacing the indexing category $\C_n$ with a simpler subcategory.

\begin{lem}\label{main-lem-1}
Let $n \geq 0$ and consider the full subcategory $\C^{0}_{n} \subseteq \C_n$  spanned by objects of the form $f: [m] \lrar [n]$ such that $f$ is \textbf{surjective}. Then the inclusion $\C^{0}_{n} \hrar \C_n$ is cofinal.
\end{lem}
\begin{proof}
We need to show that for every object $X \in \C_n$ the category
$ \C^0_{n} \times_{\C_n} {\C_n}_{/X} $
is weakly contractible. Let $X$ be the object corresponding to a morphism $g: [k] \lrar [n]$. The objects of the category $\C^0_{n} \times_{\C_n} {\C_n}_{/X}$ can be identified with commutative diagrams of the form
$$ \xymatrix{
[k] \ar_{g}[dr]\ar^{h}[rr] & & [m] \ar^{f}[dl] \\
& [n]  &\\
}$$
such that $f$ is surjective and $h$ is injective (and $g$ remains fixed). A morphism $\C^0_{n} \times_{\C_n} {\C_n}_{/X}$ between two diagrams as above as a morphism of diagrams in the opposite direction which is the identity on $[k]$ and $[n]$. A careful examination shows that the category $\C^0_{n} \times_{\C_n} {\C_n}_{/X}$ is then in fact isomorphic to the product
$$ \C^0_{n} \times_{\C_n} {\C_n}_{/X} \cong \prod_{i=0}^{n} \mcal{E}_i $$
such that
$$ \mcal{E}_i = \left\{
\begin{matrix}
\Del^{op}_s & g^{-1}(i) = \emptyset \\
{\Del^{op}_s}_{/g^{-1}(i)} & g^{-1}(i) \neq \emptyset \\
\end{matrix}\right.
$$
When $g^{-1}(i) \neq \emptyset$ then $\mcal{E}_i$ has a terminal object and so is weakly contractible. When $g^{-1}(i) = \emptyset$ then $\mcal{E}_i = \Del^{op}_s$ which is weakly contractible as well.

\end{proof}

In view of Remark~\ref{homotopy-limit} this means in particular that
\begin{equation}\label{e:c0}
\RK^+\left(X\right)_n \simeq \holim_{\C^0_n} \mcal{G}_n .
\end{equation}
We now observe the following:
\begin{lem}\label{contractible-fiber}
The category $\C^0_n$ is weakly contractible.
\end{lem}
\begin{proof}
The category $\C^0_n$ is isomorphic to $(\Del^{op}_s)^n$: the isomorphism is given by sending a surjective map $f:[m] \lrar [n]$ to the vector of linearly ordered sets $(f^{-1}(0),...,f^{-1}(n))$ considered as an object of $(\Del^{op}_s)^n$.
\end{proof}

In light of~\ref{e:c0} and Lemma~\ref{contractible-fiber} the proof of Theorem~\ref{counit} will be done once we show that for each $n$ the restricted functor ${\G_n}|_{\C^0_n}$ is homotopy-constant. This is done in the following Proposition:

\begin{prop}\label{main-lem-2}
Let $X$ be a complete marked semiSegal space. Suppose we are given a diagram
$$ \xymatrix{
[k] \ar_{g}[dr]\ar^{h}[rr] & & [m] \ar^{f}[dl] \\
& [n]  &\\
}$$
such that both $f,g$ are surjective and $h$ is injective. Then $h$ induces an equivalence
$$ h^*:X^{f}_m \x{\simeq}{\lrar} X^{g}_k .$$
\end{prop}
\begin{proof}
Since $g$ is a surjective map between simplices it admits a section $s: [n] \lrar [k]$. One then obtain a sequence
$$ X^{f}_m \x{h^*}{\lrar} X^{g}_k  \x{s^*}{\lrar} X^{\Id}_n  .$$
From the $2$-out-of-$3$ rule we see that it will be enough to prove the lemma for $k=n$ and $g=\Id$. Note that in this case $X^{\Id}_n = X_n$ and we can consider $h$ as a section of $f$.

According to Proposition~\ref{Q-anodyne} we get that the map
$ X^{\left(\Del^m,A_f\right)} \lrar X^{\left(\Del^n\right)^{\flat}} $
is a DK-equivalence. By Propositions~\ref{mapping-into-complete} and~\ref{complete-DK} this map is a levelwise equivalence. Evaluating at level $0$ we get the desired result.
\end{proof}

This finishes the proof of Theorem~\ref{counit}
\end{proof}

\subsection{ Monoidality }

Since both $\Comp$ and $\Comp_s$ are monoidal, it is natural to ask whether the Quillen equivalence $\F^+ \dashv \RK^+$ can be promoted to a weakly monoidal one in the sense of Definition~\ref{d:quillen-monoidal}. Unfortunately, this is not exactly the case. However, we seem to be in a somewhat dual situation, in which we have a \textbf{lax structure} on $\F^+$, whose structure maps are weak equivalences (this in turn determines no structure on $\RK^+$, and we do not know if $\RK^+$ carries any colax structure). Note that on the level of $\infty$-categories this does not matter - the weakly monoidal structure on $\F^+$ still induces a symmetric monoidal structure on the corresponding map of $\infty$-categories (in the suitable $\infty$-sense).

To begin, consider the adjunction
$$ \xymatrix{
\Top^{\Del_s^{\op}} \ar@<0.5ex>[r]^{\LK} & \Top^{\Del^{\op}} \ar@<0.5ex>[l]^{\F} \\
}$$
where $\F$ is the forgetful functor and $\LK$ is the left Kan extension functor. This adjunction carries a compatible lax-colax structure $(\alp_{X,Y},u),(\bet_{Z,W},v)$ (see Definition~\ref{d:monoidal-adj}) with respect to the $\otimes$ product on $\Top^{\Del_s^{\op}}$ and the Cartesian product $\times$ on $\Top^{\Del^{\op}}$. The resulting lax monoidal adjunction is in fact strongly monoidal, i.e. $v$ and $\bet_{Z,W}$ are \textbf{isomorphisms} for all $Z,W$ (in particular, the unit of the Cartesian structure, i.e. the standard $0$-simplex, can be identified with $\LK\left(\Del^0\right)$).

The isomorphisms $\bet_{Z,W}$ are completely explicit, and through them one can obtain an explicit formula for the $\alp_{X,Z}$. This in turn shows that $\alp_{X,Y}$ and $u$ actually give \textbf{marked maps}
$$ \alp^+_{X,Y}:\F^+(X) \otimes \F^+(Y) \lrar \F^+(X \otimes Y) $$
and
$$ u^+: \Del^0 \lrar \F^+(\LK(\Del^0)) $$
constituting a lax structure on $\F^+$. Now $\F^+$ in turn is the right functor in the adjunction $\F^+ \dashv \RK^+$. However, since we have a lax and not a colax structure on $\F^+$ we do not get any type of structure on $\RK^+$.

We will now proceed to show that $\alp^+_{X,Y}$ and $u^+$ are weak equivalences in $\Comp_s$ for every $X,Y \in \Top^{\Del^{\op}}$. We start with the following direct corollary of Theorem~\ref{counit}:
\begin{lem}\label{l:lk}
Let $Z$ be an (unmarked) semi-simplicial space. Then the composition of natural maps
$$ Z^{\flat} \lrar \left(\F(\LK(Z))\right)^{\flat} \lrar \F^+(\LK(Z)) $$
is a weak equivalence in $\Comp_s$.
\end{lem}
\begin{proof}
Let $(W,M)$ be a complete marked semi-simplicial space. Mapping the above composition into $W$ yields the map
$$ \Map_+\left(\F^+(\LK(Z)),W\right) \lrar \Map_+\left(Z^{\flat},W\right) .$$
By adjunction the above map can be written as
\begin{equation}\label{e:desired}
\Map\left(Z,\F(\RK^+(W))\right) \lrar \Map\left(Z,W\right)
\end{equation}
where the map is induced by the map $\F\left(\RK^+(W)\right) \lrar W$ of semi-simplicial spaces underlying the counit map. By Theorem~\ref{counit} this counit map is a marked equivalence and so the underlying map is a levelwise equivalence. Since both $\F\left(\RK^+(W)\right)$ and $W$ are Reedy fibrant we get that the map~\ref{e:desired} is a weak equivalence. Since this is true for any complete marked semi-simplicial space $W$ we get that the map
$ Z^{\flat} \lrar \F^+(\LK(Z)) $
is a weak equivalence in $\Comp_s$ as desired.
\end{proof}

Now, applying Lemma~\ref{l:lk} for $Z = \Del^0$ we obtain that
$ u^+ $
is a weak equivalence in $\Comp_s$. It is left to prove the following:
\begin{prop}\label{p:monoidality}
The natural map
$$ \alp^+_{X,Y}:\F^+(X) \otimes \F^+(Y) \lrar \F^+(X \otimes Y) $$
is a weak equivalence in $\Comp_s$ for every two simplicial spaces $X,Y$.
\end{prop}
\begin{proof}

Note that for every $n$, the left Kan extension $\LK\left(\Del^n\right)$ is the standard $n$-simplex (considered as a levelwise discrete simplicial space). Since any simplicial space is a colimit (which is simultaneously a homotopy colimit) of simplices it will be enough to prove the claim for $X = \LK\left(\Del^n\right), Y = \LK\left(\Del^m\right)$. In particular, we need to show that the lower horizontal map in the diagram
$$ \xymatrix{
\Del^n \otimes \Del^m \ar[r]\ar[d] & \F^+\left(\LK\left(\Del^n \otimes \Del^m\right)\right) \ar^{\simeq}[d] \\
\F^+\left(\LK\left(\Del^n\right)\right) \otimes \F^+\left(\LK\left(\Del^m\right)\right)  \ar[r] & \F^+\left(\LK\left(\Del^n\right) \times \LK\left(\Del^m\right)\right) \\
}$$
is a weak equivalence (to check that this diagram commutes note that the underlying diagram of semi-simplicial spaces is one of the compatibility diagram of the lax structure of $\F$ and the monoidal structure of $\LK$, see Definition~\ref{d:monoidal-adj}). Now from Lemma~\ref{l:lk} we get that upper horizontal map and the left vertical map are weak equivalences (for the left vertical map one uses the fact that in a symmetric monoidal model category the product of two weak equivalences between cofibrant objects is again a weak equivalence). Since the right vertical map is an isomorphism the result follows from the $2$-out-of-$3$ property.

\end{proof}

\begin{rem}
By using adjunction and the exponential law one sees that the lax structure on $\F^+$ induces a natural map $$ \F^+\left(\RK^+(Y)^X\right) \lrar Y^{\F^+(X)} $$
for every $X \in \Comp, Y \in \Comp_s$. Proposition~\ref{p:monoidality} then implies that this map is a weak equivalence whenever $Y$ is fibrant, i.e. a complete marked semiSegal space. In particular, if $Y = \F^{\natural}(Y')$ for some complete Segal space $Y'$ (see Remark~\ref{r:topological-adj}) then we get a sequence of weak equivalences
$$ \F^+\left((Y')^X\right) \x{\simeq}{\lrar} \F^+\left(\left(\RK^+\left(\F^{\natural}(Y')\right)\right)^X\right) \x{\simeq}{\lrar} \F^{\natural}(Y')^{\F^+(X)} \simeq \F^{\natural}(Y')^{\F^{\natural}(X)} .$$
The composition of all these equivalences can be interpreted as follows: if $\C,\D$ are two $\infty$-categories and $\ovl{\C},\ovl{D}$ their respective underlying quasi-unital $\infty$-categories, then the quasi-unital functor category $\ovl{C}^{\ovl{D}}$ is equivalent to the underlying quasi-unital $\infty$-category of the functor category $C^D$.
\end{rem}

\end{document}